\documentclass[letterpaper,11pt]{article}
\usepackage[left=1in,right=1in,top=1in,bottom=1in]{geometry}

\usepackage[T1]{fontenc}
\usepackage{amsmath, amsthm, amssymb, mathtools}
\usepackage{euscript}
\usepackage{bm}
\usepackage{graphicx}
\usepackage{algorithm}
\usepackage{algpseudocode}
\usepackage[numbers,square]{natbib}
\usepackage[colorlinks,linkcolor=blue,citecolor=blue,urlcolor=magenta,linktocpage,plainpages=false]{hyperref}
\usepackage{caption}
\usepackage{subcaption}
\usepackage{enumitem}
\usepackage{microtype}
\usepackage{array}
\usepackage{booktabs}
\usepackage [ short , c2 ]{ optidef }
\allowdisplaybreaks
\usepackage{xargs} 
\usepackage[pdftex,dvipsnames]{xcolor}  
\usepackage[colorinlistoftodos,prependcaption,textsize=tiny]{todonotes}
\newcommandx{\unsure}[2][1=]{\todo[linecolor=red,backgroundcolor=red!25,bordercolor=red,#1]{#2}}

\newcolumntype{x}[1]{>{\centering\arraybackslash\hspace{0pt}}p{#1}}
\usepackage{titlesec}

\newif\ifdarkmode
\darkmodefalse  

\definecolor{darkbg}{HTML}{2b2b2b}
\definecolor{lighttext}{HTML}{e0e0e0}
\definecolor{lightblue}{HTML}{6ba3ff}

\ifdarkmode
  \pagecolor{darkbg}
  \color{lighttext}
  \definecolor{reviewercolor}{HTML}{6ba3ff}
  \definecolor{responsecolor}{HTML}{7ec97e}
  \hypersetup{
    colorlinks=true,
    linkcolor=lightblue,
    urlcolor=lightblue,
    citecolor=lightblue
  }
\else
  \definecolor{reviewercolor}{rgb}{0.0, 0.0, 0.6}
  \definecolor{responsecolor}{rgb}{0.0, 0.5, 0.0}
\fi

\newtheorem{theorem}{Theorem}
\newtheorem{corollary}{Corollary}

\newtheorem{example}{Example}

\newcommand{\vecx}{\mathbf{x}}
\newcommand{\vecy}{\mathbf{y}}
\newcommand{\vech}{\mathbf{h}}
\newcommand{\vecc}{\mathbf{c}}
\newcommand{\vecq}{\mathbf{q}}
\newcommand{\vecd}{\mathbf{d}}
\newcommand{\vecell}{\boldsymbol{\ell}}
\newcommand{\vecu}{\mathbf{u}}
\newcommand{\vecs}{\mathbf{s}}
\newcommand{\vecw}{\mathbf{w}}
\newcommand{\vecb}{\mathbf{b}}
\newcommand{\vecpi}{\boldsymbol{\pi}}
\newcommand{\veclambda}{\boldsymbol{\lambda}}
\newcommand{\vecmu}{\boldsymbol{\mu}}

\DeclareMathOperator*{\argmin}{argmin}

\newcommand{\scenario}{k}
\newcommand{\onescenario}{\bar{k}}
\newcommand{\firststageindex}{i} 
\newcommand{\consensusvar}{\bar{\vecx}} 
\newcommand{\cvarcomp}{\bar{x}} 

\newcommand{\allfirststageindices}{\firststageindex \in [\totalfirststage]} 
\newcommand{\allscenarios}{\scenario \in [\scenarios]} 

\newcommand{\viol}{\nu} 

\newcommand{\numcontvarsone}{n_1} 
\newcommand{\numcontvarstwo}{n_2} 
\newcommand{\totalfirststage}{\numcontvarsone} 

\newcommand{\firststagespace}{\mathbb{R}_+^{\numcontvarsone}} 
\newcommand{\secondstagespace}{\mathbb{R}_+^{\numcontvarstwo}} 

\newcommand{\zSAA}{z^{2SP}} 
\newcommand{\zCG}{z^{CG}} 
\newcommand{\znac}{z^{NAC}} 

\newcommand{\zt}{\consensusvar^{(t)}}

\newcommand{\X}{\mathcal{X}}

\newcommand{\Pset}{\mathcal{P}}
\newcommand{\Pblock}[1]{\mathcal{P}_{#1}}
\newcommand{\vertex}{v} 
\newcommand{\scenprob}[1][\scenario]{p_{#1}} 

\newcommand{\candidates}{\mathcal{C}} 
\newcommand{\groupvar}{G}

\newcommand{\scenarios}{K}
\newcommand{\cgsubproblem}{\mathcal{D}_\scenario}

\newcommand{\conv}{\text{conv}}
\newcommand{\ext}{\text{ext}}
\newcommand{\startx}{\hat{\vecx}}
\newcommand{\feas}{\tilde{\vecx}} 

\newcommand{\job}{j}
\newcommand{\machine}{m}
\newcommand{\pmachines}{M}
\newcommand{\allmachines}{\bar{M}}
\newcommand{\capacity}{C}
\newcommand{\resource}{w}
\newcommand{\randomresource}{w}
\newcommand{\cost}{c}
\newcommand{\randomcost}{c}
\newcommand{\penaltycost}{r}
\newcommand{\jobmachine}{\job\machine}
\newcommand{\timestep}{t}
\newcommand{\epoch}{e}
\newcommand{\horizon}{h}
\newcommand{\lookaheadset}{H}

\newcommand{\scenarioset}{[\scenarios]}
\newcommand{\jobst}{J}

\newcommand{\jobssecond}{J_{\lookaheadset}}
\newcommand{\jobssecondk}{\jobssecond^k}
\newcommand{\jobsintersectk}{\jobst \cap \jobssecondk}

\newcommand{\partition}{\mathcal{G}} 
\newcommand{\constraintgroup}[1]{\groupvar_{#1}} 
\newcommand{\partitionblock}[1]{G_{#1}} 
\newcommand{\partindex}{\ell} 
\newcommand{\numpartitions}{r} 
\newcommand{\numconstraints}[1]{m_{#1}} 
\newcommand{\consindex}{j} 
\newcommand{\numsubdivisions}{q} 

\newcommand{\weight}[2]{w_{#1,#2}} 
\newcommand{\weightvector}[1]{\mathbf{w}_{#1}} 
\newcommand{\allweights}{\mathbf{w}} 
\newcommand{\qgradvector}[1]{\mathbf{g}_{#1}} 
\newcommand{\zweighted}[1]{z^{\text{surr}}(#1)} 
\newcommand{\zSU}{z^{SU}} 

\newcommand{\lagrmult}{\vecpi} 
\newcommand{\lagrbound}[1]{LR(#1)} 
\newcommand{\zLD}{z^{LD}} 
\newcommand{\zLP}{z^{LP}} 

\newcommand{\variant}[1]{\texttt{#1}} 

\newcommand{\papertitle}{Aggregated Column Generation for Two-Stage Stochastic Mixed-Integer Programs}

\usepackage{authblk}

\newcommand{\keywordblock}{}
\newenvironment{ackblock}{\section*{Acknowledgements}}{}
\newcommand{\coiblock}{}
\newcommand{\appendixheadings}{%
  \titleformat{\section}{\normalfont\Large\bfseries}{Appendix \thesection:}{1em}{}}

\title{\papertitle}

\author[1]{Harshit Kothari}
\author[1]{James Luedtke}
\author[2]{Luke Marshall}

\affil[1]{University of Wisconsin--Madison}
\affil[2]{Microsoft Research}
\affil[*]{Corresponding author: \texttt{hkothari2@wisc.edu}}

\date{October 13, 2026}

\begin{document}
\maketitle
%
%

\begin{abstract}
Dantzig--Wolfe reformulation applied to two-stage stochastic mixed-integer programs (SMIP) has been shown empirically to yield a formulation with strong bounds that can be solved via column generation. However, the time to solve the restricted master problem  grows significantly with the number of scenarios due to the increasing number of nonanticipativity constraints (NACs). Moreover, the associated high-dimensional dual space can cause the dual multipliers to oscillate between iterations, further slowing convergence. We propose an \emph{aggregated column generation}  framework for solving a Dantzig--Wolfe reformulation. Applied to SMIP, this method partitions the NACs into groups and replaces each group with a weighted surrogate constraint, leading to a reduced size and providing implicit stabilization by restricting the dual space. While aggregation leads to a relaxation, we develop two strategies for recovering the standard Dantzig--Wolfe reformulation bound. The first iteratively refines the partition by splitting groups that contain violated constraints. The second keeps the partition fixed and adaptively updates the aggregation weights. We prove convergence of the two approaches and study their computational performance on two benchmark problem classes.
\keywordblock
\end{abstract}

\section{Introduction}
\label{sec:introduction}

Two-stage stochastic mixed-integer programs (SMIPs) model decisions made under
uncertainty. First-stage decisions must be chosen before the uncertain
information is revealed, while second-stage recourse decisions are made after
the uncertainty is observed. These models arise in applications such as
capacity expansion \cite{ahmed2003scenario}, supply chain management
\cite{sodhi2009modeling}, and energy systems planning
\cite{wallace2003stochastic}. We study methods for computing dual bounds for
two-stage SMIPs with mixed-integer variables in both stages.
We focus on the finite-scenario setting.
Let $[\scenarios] := \{1,\ldots,\scenarios\}$ index the
scenarios, where scenario $\scenario$ corresponds to a realization of
the uncertain second-stage data
$(\mathbf{q}_\scenario,T_\scenario,W_\scenario,\mathbf{h}_\scenario)$ and has probability
$\scenprob$.

The {\it extensive form} of a two-stage SMIP can be written as follows,
\begin{mini}
    {}{\mathbf{c}^\top \mathbf{x} + \sum_{\scenario=1}^\scenarios \scenprob \mathbf{q}_\scenario^\top \mathbf{y}_\scenario}{\label{form:extensive_general}}{\zSAA = }
    \addConstraint{\mathbf{x} \in \mathcal{Q}}{}
    \addConstraint{T_\scenario \vecx + W_\scenario \vecy_\scenario}{= h_\scenario, \quad \forall\, \allscenarios}
    \addConstraint{\mathbf{x}}{\in \firststagespace, \ \vecx_{\mathcal{I}_1} \in \mathbb{Z}^{|\mathcal{I}_1|}}
    \addConstraint{\mathbf{y}_\scenario}{\in \secondstagespace,  \ \vecy_{\scenario,\mathcal{I}_2} \in
	 \mathbb{Z}^{|\mathcal{I}_2|} \quad \forall\, \allscenarios.}
\end{mini}
Here $\vecx$ represents the vector of first-stage decision variables, $\mathcal{Q}$ is a polyhedron representing constraints on the first-stage
variables, $\vecy_{\scenario}$ represents the vector of decisions that will be made
in the second stage if scenario $\scenario \in [\scenarios]$ occurs, and $\mathcal{I}_1$ and $\mathcal{I}_2$ represent
the subsets of first- and second-stage integer decision variables.
 The extensive form
becomes difficult to solve as the number of scenarios grows, motivating
decomposition methods that exploit the scenario structure.
While Benders decomposition
\cite{Benders1962,van1969shaped} is effective for problems with continuous
recourse, it is not directly applicable to SMIPs with integer recourse.

Lulli and Sen \cite{lulli2004branch} applied the Dantzig--Wolfe
framework to derive a branch-and-price method for solving multi-stage SMIP. Specializing to the two-stage case, this approach uses a reformulation in which a copy of the first-stage decision is introduced for
each scenario.
For each scenario $\scenario \in [\scenarios]$,
let $\vecx_\scenario$ denote the corresponding copy of the first-stage
decision variables. Because the first-stage decision is made before the scenario is observed, its
scenario-specific copies must agree. This requirement is enforced through
\emph{nonanticipativity constraints} (NACs). Introducing a consensus variable $\consensusvar$ gives the equivalent formulation
\begin{mini!}
    {}{\sum_{\scenario=1}^\scenarios \scenprob \Bigl(\vecc^\top \vecx_\scenario + \vecq_\scenario^\top \vecy_\scenario \Bigr)}{\label{form:saa_nac}}{\zSAA = }
    \addConstraint{\vecx_\scenario}{= \consensusvar, \quad \forall\, \allscenarios}{\label{eq:nac}}
    \addConstraint{(\vecx_\scenario, \vecy_\scenario)}{\in \cgsubproblem, \quad \forall\, \allscenarios.}{\label{eq:sscen}}
\end{mini!}
Here, for each $\allscenarios$,
\begin{align*}
\cgsubproblem = \bigl\{ \vecx \in \firststagespace, \vecy \in \secondstagespace : \ & \vecx \in \mathcal{Q},
T_\scenario \vecx + W_\scenario \vecy = \vech_\scenario,  \vecx_{\mathcal{I}_1} \in \mathbb{Z}^{|\mathcal{I}_1|}, 
 \vecy_{\mathcal{I}_2} \in \mathbb{Z}^{|\mathcal{I}_2|}
\bigr\}. 
\end{align*}
In Dantzig--Wolfe decomposition, a relaxation is obtained by
requiring
$(\vecx_\scenario, \vecy_\scenario)$ be written as a convex combination of the extreme points of the convex hull of
$\cgsubproblem$ for each $\scenario$.
Dual decomposition \cite{caroe1999dual} is a closely related approach in which a Lagrangian dual problem is obtained by relaxing the NACs \eqref{eq:nac}.
These two perspectives yield the same relaxation bound, which is at least as strong as the
linear-programming relaxation of the extensive form
\cite{lubbecke2005selected} and has been found empirically to often be very near the optimal value $\zSAA$ \cite{boland2018combining}. 
We refer to this value as the \emph{nonanticipativity bound}, denoted by $\znac$, and the focus of this paper is on
computational methods for computing this bound. 

Because the number of extreme points of the convex hull of each set
$\cgsubproblem$ may be extremely large, the Dantzig--Wolfe relaxation is typically solved by
\emph{column generation}. At each iteration, pricing subproblems identify
improving extreme points, which are added as columns to a restricted master
problem (RMP). The procedure terminates when no scenario pricing subproblem contains
a column with negative reduced cost, at which point the 
RMP attains $\znac$. When integer-optimal
solutions are required, column generation can be embedded in
branch-and-bound, yielding the branch-and-price framework for SMIPs introduced by
Lulli and Sen~\cite{lulli2004branch}.

The Dantzig--Wolfe formulation contains one NAC for each first-stage variable
and scenario, for a total of
$\numcontvarsone\cdot\scenarios$ linking constraints. The large number of NACs creates two closely related computational challenges.
First, the RMP is solved at every
column-generation iteration, so a large set of linking constraints can make
each iteration computationally expensive. Second, the corresponding
high-dimensional dual space is often highly degenerate. This degeneracy can
cause the dual multipliers used in pricing to oscillate substantially between
iterations, producing the well-known bang-bang effect and slowing
column-generation convergence \cite{briant2008comparison}. 

Following ideas introduced previously to enhance column generation for combinatorial problems with
set-partitioning structure \cite{elhallaoui2010multi,elhallaoui2005dynamic}
and packing and covering structure~\cite{costa2022stabilized}, we propose to address this issue with an \emph{aggregated column
generation} (ACG) framework that aggregates the NACs in the RMP. 
We derive these ideas generically for Dantzig--Wolfe decomposition and then specialize them
to the case of SMIPs with the NACs serving as the linking constraints.
The idea is to partition
the linking constraints into groups and replace the constraints within each group with a single weighted surrogate
constraint. This
aggregation reduces the number of linking constraints in the RMP
and restricts the dual space used for pricing. Thus, aggregation
simultaneously reduces the size of the RMP and provides an implicit
form of dual stabilization.  We develop two mechanisms for closing the gap
between the aggregated relaxation and $\znac$: \emph{partition refinement}, which
iteratively refines the partition of the linking constraints, and \emph{weight updates}, which keeps the partition
fixed and updates the aggregation weights. We then specialize these ideas to the Dantzig--Wolfe formulation
for SMIPs. 
In this setting, a natural initial partition groups all scenarios associated
with each first-stage variable, reducing the initial number of linking
constraints from
$\numcontvarsone\cdot\scenarios$
to $\numcontvarsone$.

Our work contributes to the literature on Dantzig--Wolfe formulation and column generation for solving SMIPs.
While our work builds on the scenario decomposition used in Lulli and Sen~\cite{lulli2004branch},
Silva and Wood \cite{silva2006solving} propose an alternative use of column generation/branch-and-price that exploits
instead problem-based decomposition structure (e.g., decomposing by facility in a stochastic facility location
problem).
Singh et al.~\cite{singh2009dantzig} applied Dantzig--Wolfe decomposition to a multistage
stochastic capacity-planning problem, using a particular relaxation that enables decomposition by nodes in the
scenario tree. 

More broadly, a variety of other methods for computing the nonanticipativity bound $\znac$ have been studied.
In the dual decomposition method of Car{\o}e and Schultz~\cite{caroe1999dual}, this bound is computed by optimizing the
Lagrangian dual with a bundle method \cite{kiewel1990proximity}.
Subsequent work has improved scalability of such bundle methods through parallelism, asynchrony, and specialized algorithms for optimizing the Lagrangian dual
\cite{aravena2021asynchronous,kim2019asynchronous,kim2018algorithmic,lim2021parallelizing,lubin2013parallelizing}.
Progressive hedging provides another scenario-decomposition approach based on the
nonanticipativity relaxation
for {\it convex} two-stage
stochastic programs~\cite{rockafellar1991scenarios}.
For SMIPs, progressive hedging loses its convergence
guarantees but has been shown to be useful for generating primal solutions
\cite{guo2015integration,ryan2013toward,watson2011progressive} and its dual iterates can be used to find valid bounds on the optimal value~\cite{gade2016obtaining}.
In \cite{boland2018combining}, the authors demonstrated that
combining progressive hedging with the Frank-Wolfe method
leads to a method that successfully computes the
nonanticipativity bound $\znac$.

Constraint aggregation and disaggregation techniques have a long history in mathematical optimization, e.g.,
\cite{bean1987aggregation,rogers1991aggregation}. These techniques have also been used in stochastic {\it linear}
programming \cite{birge1985aggregation,wright1994primal} and the specific idea of updating a relaxation by
adaptively partitioning a set of scenarios has
been used in \cite{van2018adaptive,espinoza2014primal,ramirez2022generalized,song2015adaptive}. However, our paper is
the first to use this idea for aggregating nonanticipativity constraints in a Dantzig--Wolfe formulation designed to yield strong
bounds for two-stage SMIPs.

Our approach is also related to the literature investigating dual stabilization methods for column generation.
These methods aim to reduce oscillation in the dual multipliers passed to the
pricing subproblems. Classical approaches
restrict dual iterates to remain near a stability center or penalize deviations
from it \cite{marsten1975boxstep,du1999stabilized}. Other approaches seek
better-centered dual solutions through interior-point or in-out stabilization
schemes \cite{gondzio2013new,pessoa2018automation,munari2013using,rousseau2007interior}, or tighten the dual feasible region using dual-optimal
inequalities \cite{amor2006dual}. 

Our work is most closely related to the previously mentioned dynamic constraint aggregation 
approach proposed for column generation of combinatorial problems with
set-partitioning structure \cite{elhallaoui2010multi,elhallaoui2005dynamic}
and packing and covering structure~\cite{costa2022stabilized}.
In \cite{elhallaoui2010multi,elhallaoui2005dynamic}
subsets of set-partitioning constraints are grouped and replaced with a representative constraint, which is then
selectively disaggregated, while also restricting the variables in the approximate problem to be consistent with the given groupings.
In
\cite{costa2022stabilized} the idea is to aggregate rows with similar dual values  in the master problem.
Different from \cite{costa2022stabilized}, our refinement strategy
updates the partition by replacing each selected aggregate with smaller
subgroups, rather than retaining the aggregate and adding constraints for
selected subsets. 
These methods motivate our use of aggregation but we describe them in a general Dantzig--Wolfe setting and adapt them to the nonanticipativity structure of two-stage SMIPs. 
In Section~\ref{sec:aggregate_constraints},
we compare our refinement strategy with an adaptation of the framework
of~\cite{costa2022stabilized} to this setting.
Furthermore, while our dynamic disaggregation method for recovering the true bound follows a strategy similar
to these works, an analog of our proposed weight update scheme has not been previously studied in this context.

\paragraph{Contributions.}
We summarize our main contributions as follows.

\begin{enumerate}[label=(\roman*)]
\item
We adapt the dynamic constraint aggregation ideas of
\cite{costa2022stabilized,elhallaoui2010multi,elhallaoui2005dynamic} to general
Dantzig--Wolfe reformulations and the
nonanticipativity constraints of two-stage SMIPs.   This process produces a
sequence of nondecreasing lower bounds that converges to the standard
column-generation bound. In many cases, it recovers this bound with far fewer
linking constraints than the fully disaggregated formulation.
\item
We develop an adaptive weighting strategy that keeps the partition fixed while
iteratively updating the weights in the aggregated constraints.
Thus, the 
RMP structure and size remain unchanged across iterations.
This provides a complementary way to strengthen the aggregated relaxation
without increasing the RMP size.
We prove that there exist weights that recover the exact column-generation bound
without partition refinement
(Theorem~\ref{thm:optimal_weights_exist}) 
and that a simple ``quasi-gradient'' method exhibits $\liminf$ convergence to this bound
(Theorem~\ref{thm:quasigradient_convergence}).

\item
We evaluate the proposed methods on two SMIP test problems.
We find that partition refinement substantially
improves column generation on these problems: it recovers the nonanticipativity bound with far
fewer linking constraints than the fully disaggregated formulation and performs
well even relative to a stabilized column-generation baseline that solves the
RMP with an interior-point method.
We also find that, while unable to fully close the gap in time comparable to the partition refinement method, the weight update method yields high-quality dual bounds much more quickly than an analogous
subgradient method applied to the Lagrangian dual.
\end{enumerate}

\paragraph{Organization.}
In Section~\ref{sec:cg_background} we derive our ideas in the context of general Dantzig--Wolfe decomposition.
In Section~\ref{sec:applications} we specialize the framework to SMIPs, and in
Section~\ref{sec:computational_results} we present the results of our computational study.

\section{Aggregated Column Generation for Dantzig--Wolfe Decomposition}
\label{sec:cg_background}

In this section we derive our results for Dantzig--Wolfe decomposition applied to
the general formulation:
\begin{mini!}
{\vecx}{\vecc^\top \vecx}{\label{eq:dw_original}}{z^* = }
\addConstraint{A \vecx}{= \vecb}{\label{eq:linking}}
\addConstraint{\vecx}{\in \X,}{}
\end{mini!}
where 
\[
\X = \{\vecx \in \mathbb{R}^n : D \vecx = \vecd,\; \vecell \le \vecx \le \vecu,\; \vecx_{\mathcal{I}} \in
\mathbb{Z}^{|\mathcal{I}|} \},
\]
and $\mathcal{I} \subseteq \{1,\dots,n\}$ is the index set of
integer-constrained variables.
In the context of the formulation \eqref{form:saa_nac} for SMIP, the constraints $A\vecx=\vecb$
correspond to the nonanticipativity constraints \eqref{eq:nac} and
the constraints $\vecx \in \X$ correspond to \eqref{eq:sscen}.
We refer to the constraints $A\vecx=\vecb$ as linking
constraints, since in the SMIP context they link the scenarios together.

In Section \ref{sec:general_cg_intro} we recall key results on Dantzig--Wolfe decomposition,
column generation, and Lagrangian relaxation that are the basis of our aggregation framework.
Then we describe the aggregated column generation approach in Section \ref{sec:aggregate_constraints}.

\subsection{Dantzig--Wolfe Decomposition and Lagrangian Relaxation Review}
\label{sec:general_cg_intro}

Dantzig--Wolfe decomposition derives a relaxation in which the set $\X$ is replaced by its
convex hull.
We assume throughout that $\X$ is bounded.
Then $\conv(\X)$ is a polytope with a finite set of extreme points
$\Pset = \ext(\conv(\X))$.
By the Minkowski--Weyl theorem \cite{schrijver1998theory}, every point in
$\conv(\X)$ can be expressed as a convex combination of these extreme points. We also assume that $\{\vecx \in
\conv(\X): A\vecx
= \vecb\}$ is non-empty.

Dantzig--Wolfe decomposition \cite{dantzig1960decomposition} replaces the feasible
set $\X$ in~\eqref{eq:dw_original} with an explicit representation of its convex hull. Writing each point of
$\conv(\X)$ as a convex combination of the extreme points $\Pset$, with weights
$\lambda_\vertex$ for $\vertex \in \Pset$, gives the following relaxation of
\eqref{eq:dw_original}:
\begin{mini!}
{}{\sum_{\vertex \in \Pset} (\vecc^\top \vecx_\vertex)\lambda_\vertex}{\label{eq:dw_extended}}{\zCG =}
\addConstraint{\sum_{\vertex \in \Pset} (A \vecx_\vertex)\lambda_\vertex}{= \vecb \label{eq:complic}}{}
\addConstraint{\sum_{\vertex \in \Pset} \lambda_\vertex}{= 1 \label{eq:convexity_constraint}}{}
\addConstraint{\lambda}{\ge 0.}{}
\end{mini!}
Equivalently, the same bound can be written directly in the original $\vecx$-space as
\begin{equation}
\label{eq:dw_convhull}
\zCG = \min \bigl\{ \vecc^\top \vecx : A\vecx = \vecb,\; \vecx \in \conv(\X) \bigr\}.
\end{equation}
Because $\conv(\X)$ is contained in the standard LP relaxation of $\X$, the
Dantzig--Wolfe relaxation is at least as strong as the LP relaxation of
\eqref{eq:dw_original}: $\zCG \ge \zLP$, where $\zLP$ denotes the LP relaxation
bound.

Because $|\Pset|$ is usually intractably large, the Dantzig--Wolfe relaxation \eqref{eq:dw_extended} is solved using
column generation.
At each iteration, a restricted master problem is solved which has the form of \eqref{eq:dw_extended} but uses only a subset of columns
$\Pset' \subseteq \Pset$. We assume availability of an initial subset $\Pset'$ which makes the initial RMP feasible. 
This restricted column set may not be sufficient to solve the full master
problem \eqref{eq:dw_extended}, so column generation searches for an extreme point of $\conv(\X)$ that
can improve the RMP objective.
Let $\vecpi^* \in \mathbb{R}^m$ and $\pi_0^* \in \mathbb{R}$ denote optimal dual
multipliers for the linking constraints \eqref{eq:complic} and the convexity constraint \eqref{eq:convexity_constraint} in
the current RMP.
These multipliers determine the reduced cost of a column, which is minimized over all
candidate columns by solving the pricing subproblem
\begin{equation}
\label{eq:dw_pricing}
\min \Bigl\{ \vecc^\top \vecx - (\vecpi^*)^\top A \vecx - \pi_0^* : \vecx \in \X \Bigr\}.
\end{equation}
If the optimal value of~\eqref{eq:dw_pricing} is negative, the corresponding vertex has negative reduced cost and is added to the RMP.
If no such vertex exists, the current RMP solution is optimal for~\eqref{eq:dw_extended}, thereby certifying attainment of $\zCG$.

To recover integer-optimal solutions of~\eqref{eq:dw_original}, column
generation can be embedded within branch-and-bound, yielding a branch-and-price
algorithm \cite{barnhart1998branch}.
In this paper, we focus on computing the Dantzig--Wolfe relaxation bound.


Lagrangian relaxation gives the dual counterpart of the Dantzig--Wolfe relaxation.
Instead of enforcing the linking constraints in a master problem, it dualizes
them and penalizes their violation in the objective \cite{fisher1981lagrangian}.

For a multiplier vector $\lagrmult \in \mathbb{R}^m$,
the Lagrangian relaxation of~\eqref{eq:dw_original} is 
\begin{equation}
\label{eq:lagrangian_subproblem}
\lagrbound{\lagrmult}
  \coloneqq \min_{\vecx \in \X} \Bigl\{ \vecc^\top \vecx + \lagrmult^\top (\vecb - A\vecx) \Bigr\}
  = \lagrmult^\top \vecb + \min_{\vecx \in \X} \Bigl\{ (\vecc - A^\top \lagrmult)^\top \vecx \Bigr\}.
\end{equation}
For every multiplier $\lagrmult \in \mathbb{R}^m$, $\lagrbound{\lagrmult}$ is a valid lower bound on the optimal value $z^*$.
Maximizing this bound over all multipliers gives
the Lagrangian dual
\begin{equation}
\label{eq:lagrangian_dual}
\zLD \coloneqq \max_{\lagrmult \in \mathbb{R}^m} \lagrbound{\lagrmult}.
\end{equation}

Lagrangian relaxation and Dantzig--Wolfe decomposition provide dual and primal views of the same relaxation.
Under
the standard assumptions ensuring strong duality for the resulting linear program,
the two bounds coincide:
\begin{equation}
\label{eq:lagrangian_dw_equivalence}
\zLD = \zCG.
\end{equation}
Thus, an optimal dual solution for the linking constraints in the Dantzig--Wolfe relaxation is also an optimal multiplier vector for the Lagrangian dual.
It also gives a dual interpretation of pricing: a pricing subproblem searches for a column with negative reduced cost, which corresponds to a violated dual constraint.
Column generation can therefore be interpreted as a cutting-plane method for the Lagrangian dual \cite{lubbecke2005selected}.

The Lagrangian dual function $\lagrbound{\cdot}$ is concave and piecewise linear, but generally nondifferentiable.
A standard way to maximize it is to update the multipliers using residuals of the relaxed linking constraints
in a subgradient method \cite{fisher1981lagrangian}. 

%

\subsection{Aggregated Column Generation}
\label{sec:aggregate_constraints}

In applications such as two-stage SMIP, the number of  linking constraints $m$ in \eqref{eq:complic} can be very
large, leading to a large RMP and potentially contributing to slow convergence.
This motivates the aggregated column generation approach, in which the number of linking constraints in the RMP is
reduced  by partitioning the $m$ linking constraints of~\eqref{eq:dw_original} into $\numpartitions \ll m$ groups
and replacing the constraints in each group with a single aggregated linking constraint.
The resulting RMP has $\numpartitions$ linking constraints rather than $m$.
The partition therefore determines the size of the aggregated RMP and the dimension of the associated dual space.

Let $\partition = \{\constraintgroup{1}, \ldots, \constraintgroup{\numpartitions}\}$ be a partition of the
linking-constraint index set $[m] := \{1,\ldots,m\}$.
For each group $\constraintgroup{\partindex}$, with $\partindex \in [\numpartitions]$, let $\numconstraints{\partindex}$ denote the number of constraints in the group.
Let
$A^\partindex \in \mathbb{R}^{\numconstraints{\partindex} \times n}$ and
$\vecb^\partindex \in \mathbb{R}^{\numconstraints{\partindex}}$
denote the corresponding coefficient rows and right-hand sides, where the rows of $A^\partindex$ and $\vecb^\partindex$ are indexed by the sets
$\constraintgroup{\partindex}$ for $\partindex \in [\numpartitions]$.
A weight vector
$\weightvector{\partindex} = \{\weight{\partindex}{j}\}_{j \in \constraintgroup{\partindex}} \in \mathbb{R}^{\numconstraints{\partindex}}$
specifies how the constraints in group $\constraintgroup{\partindex}$ are combined.
For each $\partindex \in [\numpartitions]$, we replace the constraints in $\constraintgroup{\partindex}$ by the single \emph{weighted aggregated constraint}
\begin{equation}
\label{eq:weighted_agg_general}
\weightvector{\partindex}^\top A^\partindex \vecx = \weightvector{\partindex}^\top \vecb^\partindex,
\qquad \forall\, \partindex \in [\numpartitions].
\end{equation}
Substituting the Minkowski--Weyl representation,
$\vecx = \sum_{\vertex \in \Pset} \vecx_\vertex \lambda_\vertex$,
gives the corresponding aggregated constraints in the RMP:
\begin{equation}
\label{eq:weighted_agg_master}
\sum_{\vertex \in \Pset} \bigl(\weightvector{\partindex}^\top A^\partindex \vecx_\vertex\bigr) \lambda_\vertex =
\weightvector{\partindex}^\top \vecb^\partindex,
\qquad \forall\, \partindex \in [\numpartitions].
\end{equation}
The resulting \emph{aggregated master problem} replaces the $m$ linking constraints of~\eqref{eq:dw_extended} with $\numpartitions$ aggregated constraints.
Because each column-generation iteration solves this smaller master LP, choosing $\numpartitions \ll m$ can
substantially reduce the time to solve the RMP.

Since the aggregated constraints are implied by the original constraints,
the aggregated formulation is a relaxation and therefore yields a lower bound on~$\zCG$.
We use two strategies to close the gap between the aggregated bound and~$\zCG$:
\begin{itemize}
  \item \textbf{Partition refinement} (Section~\ref{sec:refinement_general}):
    Fix all weights to one ($\weightvector{\partindex} = \mathbf{1}$)
    and iteratively split groups that contain violated individual constraints.
    This progressively tightens the relaxation.
  \item \textbf{Weight updates} (Section~\ref{sec:adaptive_weights}):
    Fix the partition and iteratively update the weight vectors using constraint residuals.
    The number of linking constraints stays fixed.
\end{itemize}
We refer to both of these methods as {\it aggregated column generation} (ACG) since they are both based on solving an
RMP with aggregated linking constraints.

The partition refinement strategy is closely related to the dynamic separation of aggregated rows (Dyn-SAR) method
of~\cite{costa2022stabilized}.
Dyn-SAR was developed for set-partitioning, packing, and covering structures. We describe the partition refinement approach
in a generic Dantzig--Wolfe context and in Section~\ref{sec:applications} we specialize it to the
two-stage SMIP context.
Partition refinement differs somewhat from Dyn-SAR in how the disaggregation is handled: when a group is split, we
replace the parent aggregate with its children, whereas the Dyn-SAR scheme retains previously generated aggregates while
adding rows for selected violations. By keeping the previous rows Dyn-SAR allows the updated RMP to be solved by
dual simplex starting from the previous RMP basis. In contrast, our approach leads to smaller growth in the size of the
RMP but loses this warm-start capability.

\subsubsection{Partition Refinement}
\label{sec:refinement_general}

A natural starting point is the unit-weight case of~\eqref{eq:weighted_agg_general}, where
$\weight{\partindex}{\consindex} = 1$ for every $\consindex \in \constraintgroup{\partindex}$.
The aggregated constraints become
\begin{equation}
\label{eq:agg_uniform}
\mathbf{1}^\top A^\partindex \vecx = \mathbf{1}^\top \vecb^\partindex,
\qquad \forall\, \partindex \in [\numpartitions].
\end{equation}
Column generation on this aggregated master follows the standard 
procedure: the RMP is solved to obtain dual multipliers $\tilde{\pi}_\partindex$ for $\partindex \in [\numpartitions]$. Since the
constraints are equally weighted, these are then extended to duals in the full dual space via
$\pi^{\text{eff}}_{j} = \tilde{\pi}_{\partindex}$ for $\partindex \in [\numpartitions]$ and $j \in \constraintgroup{\partindex}$, which are then passed to
the pricing subproblem~\eqref{eq:dw_pricing}. 


Aggregation reduces the number of dual variables in the RMP from $m$ to $\numpartitions$
and provides implicit stabilization:
multipliers within the same group share a common aggregate multiplier, eliminating
within-group dual oscillations.
The trade-off is that~\eqref{eq:agg_uniform} enforces only the sum of the constraints
in each group, not the individual constraints themselves.
Positive and negative residuals within a group can cancel, so the resulting bound may
be weaker than~$\zCG$.
Refining the partition tightens the relaxation monotonically.
If $\partition'$ refines $\partition$, meaning that every group of $\partition'$ is contained in some group of $\partition$, then the formulation under $\partition'$ is at least as tight as the formulation under $\partition$.
When every group is a singleton, the aggregated formulation coincides with the full disaggregated RMP and
yields~$\zCG$.
This monotonicity motivates an iterative
strategy: start from a coarse partition, split groups with violated constraints, and
stop once the disaggregated constraints are satisfied.

Let $\vecx^*$ be an optimal solution of the current aggregated RMP (after completion of the column generation process). 
We identify the groups that
contain at least one violated constraint:
\[
\candidates =
\bigl\{
\groupvar \in \partition :
\exists\, \consindex \in \groupvar
\text{ with } A_\consindex \vecx^* \neq \vecb_\consindex
\bigr\}.
\]
If $\candidates = \emptyset$, then the aggregated solution satisfies all original linking constraints, so the current
bound equals $\zCG$.
Otherwise, we select a subset
$\mathcal{S} \subseteq \candidates$ for refinement.
To limit growth in the RMP, we split only a fixed number or fraction of the candidate
groups at each iteration. We rank the candidates by a measure $\viol_{\groupvar}(\vecx^*)$ of total
violation and refine those with the largest values first.
Two such measures are the total absolute violation
$\viol_{\groupvar}(\vecx^*) = \sum_{\consindex \in \groupvar}
|A_\consindex \vecx^* - \vecb_\consindex|$
and the mean squared violation $\viol_{\groupvar}(\vecx^*)=1/|\groupvar| \sum_{\consindex \in \groupvar}
(A_\consindex \vecx^* - \vecb_\consindex)^2$.

For each selected group, a natural subdivision rule is to split constraints by the sign of their residual
$r_\consindex := A_\consindex \vecx^* - \vecb_\consindex$.
This creates two subgroups, corresponding to negative and positive residuals. The components with residuals
equal to zero can either be separated into their own subgroup or included with the negative or positive subgroups.
The split separates constraints with residuals of opposite signs, preventing their violations from canceling in later iterations.
Algorithm~\ref{alg:agg_cg_refinement} summarizes the resulting refinement procedure.
At each outer iteration, we solve the current aggregated RMP by column generation and then identify groups that contain violated individual constraints.
If no violations remain, the current bound equals $\zCG$ and the algorithm terminates.
Otherwise, the selected candidate groups are split and the process repeats.

In the worst case, the algorithm refines every group to a singleton and recovers the
full disaggregated formulation. In our experiments, however, $\zCG$ is often attained with
$|\partition|$ substantially smaller than $m$; see
Section~\ref{sec:computational_results}. 
Early iterations still use a coarse partition, so the RMP is small and the
dual space is low-dimensional. The restricted dual space reduces the oscillations that
often slow standard column generation, while the coarse aggregates can still capture
large improvements in the bound. As a result, the method can make rapid early
progress before adding the finer constraints needed to recover $\zCG$.

\begin{algorithm}[H]
\caption{Aggregated Column Generation with Partition Refinement}
\label{alg:agg_cg_refinement}
\begin{algorithmic}[1]
\State \textbf{Input:} Initial partition $\partition$ of $[m]$.
\Repeat
    \State Solve aggregated LP~\eqref{eq:agg_uniform} via column generation;
           let $\vecx^*$ be optimal.
    \State Compute $\candidates \gets \{\groupvar \in \partition : \exists\, \consindex \in \groupvar,\ A_\consindex
	 \vecx^* \neq \vecb_\consindex\}$.
    \If{$\candidates = \emptyset$}
        \State \textbf{break} \Comment{All individual constraints satisfied; bound equals $\zCG$.}
    \EndIf
    \State Select non-empty $\mathcal{S} \subseteq \candidates$ for refinement (e.g., largest $\viol_{\groupvar}(\vecx^*)$).
    \For{each $\groupvar \in \mathcal{S}$}
        \State Subdivide $\groupvar$ into disjoint subsets $\{\groupvar_1, \ldots, \groupvar_\numsubdivisions\}$
               (e.g., by sign of residual).
        \State $\partition \gets (\partition \setminus \{\groupvar\}) \cup \{\groupvar_1, \ldots, \groupvar_\numsubdivisions\}$.
    \EndFor
\Until{convergence}
\State \textbf{Output:} Optimal solution $\vecx^*$ and dual bound $\zCG$.
\end{algorithmic}
\end{algorithm}

\subsubsection{Weight Updates}
\label{sec:adaptive_weights}

Partition refinement can recover $\zCG$, but it increases the number of RMP constraints as groups are split.
An alternative is to keep the partition fixed at $\numpartitions$ groups and improve the bound by updating the weight
vectors, leaving the RMP structure unchanged.

Consider a fixed partition
$\partition = \{\constraintgroup{1}, \ldots, \constraintgroup{\numpartitions}\}$ of
$[m]$ with weight vectors
$\{\weightvector{\partindex}\}_{\partindex=1}^{\numpartitions}$.
Each group $\partindex \in [\numpartitions]$ contributes one aggregated constraint~\eqref{eq:weighted_agg_general} to
the RMP, so it has one dual variable $\pi_\partindex$.
Let $\pi_0$ denote the dual multiplier for the convexity constraint.
The pricing subproblem becomes
\begin{equation}
\label{eq:weighted_pricing_general}
\min \Bigl\{ \vecc^\top \vecx
    - \sum_{\partindex=1}^{\numpartitions}
      \pi_\partindex \weightvector{\partindex}^\top A^\partindex \vecx
    - \pi_0 : \vecx \in \X \Bigr\}.
\end{equation}
Expanding the weighted aggregates shows that each original constraint
$A^\partindex_\consindex \vecx = \vecb^\partindex_\consindex$ receives the effective dual
multiplier
\begin{equation}
\label{eq:effective_dual_general}
\pi_\consindex^{\text{eff}}
\coloneqq
\tilde{\pi}_\partindex \weight{\partindex}{\consindex},
\qquad
\partindex \in [\numpartitions],\;
\consindex \in \constraintgroup{\partindex}.
\end{equation}
Thus, although the RMP contains only $\numpartitions$ dual variables, the weights
induce $m$ effective multipliers, one for each original constraint. Updating the
weights changes the pricing signals associated with individual constraints
while retaining a compact RMP.

We next formulate the choice of weights as an optimization problem. Let
$\allweights = (\weightvector{1}, \ldots, \weightvector{\numpartitions})
\in \mathbb{R}^m$ denote the vector obtained by concatenating the group weights. For
fixed weights $\allweights$, replacing the $m$ linking constraints with the $\numpartitions$ weighted aggregates, while
keeping the convexified set $\vecx \in \conv(\X)$, gives the generalized surrogate relaxation value
\begin{equation}
\label{eq:aggregated_lp_weighted}
\zweighted{\allweights}
  =
  \min \bigl\{ \vecc^\top \vecx :
      \weightvector{\partindex}^\top A^\partindex \vecx
      =
      \weightvector{\partindex}^\top \vecb^\partindex\;
      \ \forall\, \partindex \in [\numpartitions],\;
      \vecx \in \conv(\X)
  \bigr\}.
\end{equation}
Every point feasible for the disaggregated relaxation is feasible for this aggregated
relaxation, so $\zweighted{\allweights} \le \zCG$ for every $\allweights$.
The corresponding generalized surrogate dual problem is
\begin{equation}
\label{eq:surrogate_dual_problem}
\zSU \coloneqq \max_{\allweights \in \mathbb{R}^m}
\zweighted{\allweights}.
\end{equation}
It seeks the strongest bound obtainable with the fixed partition. By construction, $\zSU \leq \zCG$.
The surrogate dual \eqref{eq:aggregated_lp_weighted} is  closely related to the dual used in
\cite{glover1975surrogate,muller2020generalized}. However, our context is different and simpler -- our surrogate
relaxation \eqref{eq:aggregated_lp_weighted} is a 
convex problem (using column generation to handle the convex constraint $\vecx \in \conv(\X)$), whereas the surrogate
relaxation \cite{glover1975surrogate,muller2020generalized} is a nonconvex problem (e.g., containing integer variables).
As a result, we find that the special structure of our surrogate dual
enables it to be solved by a simple ``quasi-gradient'' method that would not apply in the nonconvex setting.

The next theorem shows that the generalized surrogate dual recovers the full column-generation bound.
In particular, the optimal Lagrangian dual multipliers provide one choice of optimal aggregation weights.

\begin{theorem}
\label{thm:optimal_weights_exist}
Let $\vecpi^* \in \mathbb{R}^m$ be an optimal dual solution of the Lagrangian dual \eqref{eq:lagrangian_dual}.
For each group $\partindex \in [\numpartitions]$, let $\vecpi^{*,\partindex} \in
\mathbb{R}^{\numconstraints{\partindex}}$ denote the subvector of $\vecpi^*$ associated with that group.
Then $\allweights^* = \vecpi^*$
(i.e., $\weightvector{\partindex}^* = \vecpi^{*,\partindex}$ for each $\partindex \in [\numpartitions]$) is a maximizer of the generalized surrogate dual~\eqref{eq:surrogate_dual_problem}, and
\begin{equation}
\label{eq:optimal_weights_formula}
\zSU = \zweighted{\vecpi^*} = \zCG.
\end{equation}
\end{theorem}

\begin{proof}

Since $\zweighted{\vecpi^*} \le \zSU \le \zCG$ are immediate, it suffices to prove that $\zweighted{\vecpi^*} \ge \zCG$.

Evaluating the Lagrangian relaxation~\eqref{eq:lagrangian_subproblem} at $\vecpi^*$ and using
$\zCG = \zLD$ from~\eqref{eq:lagrangian_dw_equivalence} gives
\begin{equation}
\label{eq:lagrangian_at_pistar}
\zCG
= \sum_{\partindex=1}^{\numpartitions} (\vecpi^{*,\partindex})^\top \vecb^\partindex
  + \min_{\vecx \in \conv(\X)} \Bigl\{ \vecc^\top \vecx
    - \sum_{\partindex=1}^{\numpartitions} (\vecpi^{*,\partindex})^\top A^\partindex \vecx \Bigr\}.
\end{equation}
Let $\feas$ be any feasible solution of~\eqref{eq:aggregated_lp_weighted}
with $\weightvector{\partindex} = \vecpi^{*,\partindex}$.
Then $\feas \in \conv(\X)$ and
$(\vecpi^{*,\partindex})^\top A^\partindex \feas = (\vecpi^{*,\partindex})^\top \vecb^\partindex$
for all $\partindex \in [\numpartitions]$.
Substituting the aggregated-constraint feasibility into the Lagrangian objective:
\begin{align*}
\vecc^\top \feas
&= \Bigl( \vecc^\top \feas - \sum_{\partindex=1}^{\numpartitions} (\vecpi^{*,\partindex})^\top A^\partindex \feas \Bigr) +
\sum_{\partindex=1}^{\numpartitions} (\vecpi^{*,\partindex})^\top \vecb^\partindex \\
&\ge \min_{\vecx \in \conv(\X)} \Bigl\{ \vecc^\top x - \sum_{\partindex=1}^{\numpartitions} (\vecpi^{*,\partindex})^\top
A^\partindex \vecx \Bigr\} + \sum_{\partindex=1}^{\numpartitions} (\vecpi^{*,\partindex})^\top \vecb^\partindex
= \zCG.
\end{align*}
Since every feasible solution of~\eqref{eq:aggregated_lp_weighted} with weights $\vecpi^*$
has objective value at least $\zCG$, taking the minimum gives $\zweighted{\vecpi^*} \ge \zCG$.
Combined with $\zweighted{\vecpi^*} \leq \zSU \leq \zCG$, this implies
\[
\zweighted{\vecpi^*} = \zSU = \zCG,
\]
which completes the proof.
\end{proof}


Theorem~\ref{thm:optimal_weights_exist} shows that the aggregated formulation with
$\numpartitions$ constraints can attain the full $m$-constraint bound when the weights
equal an optimal dual multiplier vector $\vecpi^*$. Because $\vecpi^*$ is obtained from the
optimal Lagrangian dual, it is not available in advance.
Directly optimizing the surrogate dual is also not straightforward:
$\zweighted{\allweights}$ need not be concave, or even quasiconcave, because the feasible set
in~\eqref{eq:aggregated_lp_weighted} depends on the weights.
Example~\ref{ex:surrogate_nonconcave} illustrates this behavior.

\begin{example}[Non-quasiconcavity of the surrogate dual]
\label{ex:surrogate_nonconcave}
Consider a three-scenario problem with first-stage copies $x_1, x_2, x_3$,
consensus variable $\cvarcomp$, and second-stage recourse variables $y_1, y_2, y_3$ for
scenarios~1-3.
The NACs are $x_j = \cvarcomp$, $j=1,2,3$:
\begin{equation}
\label{eq:lp}
\begin{aligned}
&\max && x_1\\
&\text{s.t.} && x_j = \cvarcomp, \quad j=1,2,3 \\
&&& x_1 + y_1 \le 14,
\ x_2+y_2 \le 12, \ x_3 + y_3 \le 10 ,\\
&&& x_j \geq 0, y_j \geq 0, \quad j=1,2,3 
\end{aligned}
\end{equation}
The surrogate relaxation with weights $\allweights = (w_1, w_2,w_3)$ is:
\begin{equation}
\label{eq:inner-lp}
\begin{aligned}
\zweighted{\allweights}\;:=\;&\max && x_1\\
&\text{s.t.} && w_1 x_1 + w_2 x_2 + w_3 x_3 \;=\; (w_1 + w_2 + w_3)\cvarcomp, \,\\
&&& x_1 + y_1 \le 14,
\ x_2+y_2 \le 12, \ x_3 + y_3 \le 10 ,\\
&&& x_j \geq 0, y_j \geq 0, \quad j=1,2,3 .
\end{aligned}
\end{equation}
Take \(\allweights^{(1)}=(1,-1,0)\) and \(\allweights^{(2)}=(0,1,-1)\), with midpoint
\(\bar{\allweights}=\tfrac12(\allweights^{(1)}+\allweights^{(2)})=(1/2,0,-1/2)\).
Evaluating \eqref{eq:inner-lp} at each gives
\[
\zweighted{\allweights^{(1)}}=12,\qquad
\zweighted{\allweights^{(2)}}=14,\qquad
\zweighted{\bar{\allweights}}=10.
\]
Since \(\zweighted{\bar{\allweights}}=10 < \min\{\zweighted{\allweights^{(1)}},\zweighted{\allweights^{(2)}}\}= 12 \),
the function \(\zweighted{\cdot}\) is not quasiconcave.
\end{example}

The lack of (quasi)concavity means that residuals of the original constraints are
not supergradients.
Despite this, we find that using the residuals as ascent-like directions within a weight update scheme leads to a
(lim inf) convergent algorithm.
Algorithm~\ref{alg:quasigradient_weighted} summarizes the procedure.
Given weights $\allweights^{(t)}$ at iteration $t$,
the weighted aggregated LP~\eqref{eq:aggregated_lp_weighted}
is solved via column generation to obtain an optimal solution $\vecx^{(t)}$
and objective value $\zweighted{\allweights^{(t)}}$ (line~4).
For each group $\partindex \in [\numpartitions]$, the residual vector
\begin{equation}
\label{eq:quasigradient_vector}
\qgradvector{\partindex}^{(t)} = \vecb^\partindex - A^\partindex \vecx^{(t)} \in \mathbb{R}^{\numconstraints{\partindex}}
\end{equation}
measures these per-constraint violations (line~6).
The weights are updated in the residual direction (line~7):
\begin{equation}
\label{eq:weight_update_vector}
\weightvector{\partindex}^{(t+1)} = \weightvector{\partindex}^{(t)} + \alpha_t\, \qgradvector{\partindex}^{(t)},
\qquad \forall\, \partindex \in [\numpartitions],
\end{equation}
where $(\alpha_t)_{t \geq 0}$ is a step-size sequence satisfying the
Robbins--Monro conditions~\eqref{eq:step_size_conditions}
(e.g., $\alpha_t = \alpha_0/(t+1)$).
Because the residual directions are not supergradients (the function is not concave), we
refer to them as
\emph{quasi-gradients}.
The algorithm terminates when
$\|\qgradvector{\partindex}^{(t)}\|_2 \le \epsilon$ for all $\partindex \in [\numpartitions]$ (i.e., the aggregated RMP
solution is approximately feasible)
or when a maximum number of iterations is reached.

The scheme mirrors the
subgradient method for the Lagrangian dual: each iteration computes per-constraint
violations and uses them to adjust the weights. The difference is that in this scheme the residuals are computed based
on the solution of the aggregated RMP (solved via column generation), whereas in the subgradient method applied to the Lagrangian dual the
residuals are based on the solutions of the Lagrangian relaxation subproblems.


\begin{algorithm}[H]
\caption{Quasi-Gradient Method for Weighted Aggregation}
\label{alg:quasigradient_weighted}
\begin{algorithmic}[1]
\State \textbf{Input:} Partition $\{\partitionblock{1},\ldots,\partitionblock{\numpartitions}\}$ of $[m]$
       with constraint data $\{A^\partindex, \vecb^\partindex\}_{\partindex=1}^\numpartitions$;
       step-size parameter $\alpha_0 > 0$; tolerance $\epsilon > 0$.
\State \textbf{Initialize:} $\weight{\partindex}{\consindex}^{(0)} = 1$ for all $\partindex, \consindex$;
       $t \gets 0$.
\Repeat
    \State Solve~\eqref{eq:aggregated_lp_weighted} via column generation with weights
           $\{\weightvector{\partindex}^{(t)}\}$; let $\vecx^{(t)}$ be optimal.
    \For{$\partindex = 1$ to $\numpartitions$}
        \State $\qgradvector{\partindex}^{(t)} \gets \vecb^\partindex - A^\partindex \vecx^{(t)}$
        \State $\weightvector{\partindex}^{(t+1)} \gets
               \weightvector{\partindex}^{(t)} + \alpha_t\, \qgradvector{\partindex}^{(t)}$
    \EndFor
    \State $t \gets t + 1$
\Until{$\|\qgradvector{\partindex}^{(t-1)}\|_2 \le \epsilon$ for all $\partindex$,
       or maximum iterations reached.}
\State \textbf{Output:} Weights $\{\weightvector{\partindex}^{(t-1)}\}$ and dual bound $\zweighted{\allweights^{(t-1)}}$.
\end{algorithmic}
\end{algorithm}

Despite the lack of concavity, the following theorem shows that the generated bounds attain $\zCG$ in the limit.

\begin{theorem}[Convergence of Quasi-Gradient Method]
\label{thm:quasigradient_convergence}
Let $\X$ be compact
and let the step-size sequence $\{\alpha_t\}_{t=0}^\infty$ satisfy the Robbins--Monro conditions:
\begin{equation}
\label{eq:step_size_conditions}
\sum_{t=0}^\infty \alpha_t = \infty
\quad \text{and} \quad
\sum_{t=0}^\infty \alpha_t^2 < \infty.
\end{equation}
Then the sequence of bounds $\{\zweighted{\allweights^{(t)}}\}_{t=0}^\infty$
produced by Algorithm~\ref{alg:quasigradient_weighted} with the stopping conditions removed satisfies
\begin{equation}
\label{eq:convergence_result}
\liminf_{t \to \infty}  \zweighted{\allweights^{(t)}}  =  \zCG.
\end{equation}
\end{theorem}

\begin{proof}
By Theorem~\ref{thm:optimal_weights_exist},
there exist optimal weights $\allweights^* = \vecpi^* \in \mathbb{R}^m$
such that $\zweighted{\allweights^*} = \zCG$.
Define the optimality gap at iteration $t$ as $\Delta_t \coloneqq \zCG - \zweighted{\allweights^{(t)}} \ge 0$.
Since $\vecx^{(t)}$ is feasible for the aggregated formulation~\eqref{eq:aggregated_lp_weighted},
it satisfies every weighted aggregated constraint:
\begin{equation}
\label{eq:orthogonality_condition}
\weightvector{\partindex}^{(t)\top} \qgradvector{\partindex}^{(t)} = 0,
\qquad \forall\, \partindex \in [\numpartitions],
\end{equation}
where $\qgradvector{\partindex}^{(t)} = \vecb^\partindex - A^\partindex \vecx^{(t)}$
is the residual vector for group $\partindex$.
Summing over groups gives
\begin{equation}
\label{eq:weight_orthogonality}
\sum_{\partindex=1}^\numpartitions \weightvector{\partindex}^{(t)\top} \qgradvector{\partindex}^{(t)} = 0.
\end{equation}

Since $\zCG = \zLD$ and the Lagrangian dual value at $\allweights^* = \vecpi^*$ equals $\zLD$,
\[
\zCG = \sum_{\partindex=1}^\numpartitions \weightvector{\partindex}^{*\top} \vecb^\partindex + \min_{\vecx \in \X} \Bigl\{
\vecc^\top \vecx - \sum_{\partindex=1}^\numpartitions \weightvector{\partindex}^{*\top} A^\partindex \vecx \Bigr\}.
\]
Substituting $\vecx^{(t)} \in \X$ in place of the minimizer gives an upper bound:
\begin{align}
\zCG
&\le \sum_{\partindex=1}^\numpartitions \weightvector{\partindex}^{*\top} \vecb^\partindex + \vecc^\top \vecx^{(t)} -
\sum_{\partindex=1}^\numpartitions \weightvector{\partindex}^{*\top} A^\partindex \vecx^{(t)} \notag \\
&= \zweighted{\allweights^{(t)}} + \sum_{\partindex=1}^\numpartitions \weightvector{\partindex}^{*\top} \qgradvector{\partindex}^{(t)}. \notag 
\end{align}
Rearranging yields
\begin{align}
\Delta_t
&\le \sum_{\partindex=1}^\numpartitions \weightvector{\partindex}^{*\top} \qgradvector{\partindex}^{(t)}. \label{eq:dual_gap_bound}
\end{align}

Expanding the squared distance after the weight update~\eqref{eq:weight_update_vector}:
\begin{align*}
\| \allweights^{(t+1)} - \allweights^* \|^2
&= \| \allweights^{(t)} + \alpha_t \qgradvector{}^{(t)} - \allweights^* \|^2 \\
&= \| \allweights^{(t)} - \allweights^* \|^2
+ 2\alpha_t \sum_{\partindex=1}^\numpartitions (\weightvector{\partindex}^{(t)} - \weightvector{\partindex}^{*})^\top \qgradvector{\partindex}^{(t)}
+ \alpha_t^2 \| \qgradvector{}^{(t)} \|^2.
\end{align*}
The cross term simplifies using~\eqref{eq:weight_orthogonality} and~\eqref{eq:dual_gap_bound}:
\[
\sum_{\partindex=1}^\numpartitions (\weightvector{\partindex}^{(t)} - \weightvector{\partindex}^{*})^\top \qgradvector{\partindex}^{(t)}
= \sum_{\partindex=1}^\numpartitions \weightvector{\partindex}^{(t)\top} \qgradvector{\partindex}^{(t)}
- \sum_{\partindex=1}^\numpartitions \weightvector{\partindex}^{*\top} \qgradvector{\partindex}^{(t)}
\le 0 -\Delta_t = - \Delta_t.
\]
Since $\X$ is compact, the residuals are uniformly bounded:
there exists $M > 0$ such that $\sum_{\partindex=1}^\numpartitions \| \qgradvector{\partindex}^{(t)} \|_2^2 \le M$ for all $t$.
Combining these bounds:
\begin{align*}
\| \allweights^{(t+1)} - \allweights^* \|^2
&\le \| \allweights^{(t)} - \allweights^* \|^2 - 2\alpha_t \Delta_t + \alpha_t^2 M, \\
\Rightarrow \quad 2\alpha_t \Delta_t
&\le \| \allweights^{(t)} - \allweights^* \|^2 - \| \allweights^{(t+1)} - \allweights^* \|^2 + \alpha_t^2 M.
\end{align*}

Summing from $t = 0$ to $T$, the distance terms telescope:
\[
2 \sum_{t=0}^T \alpha_t \Delta_t
\le \| \allweights^{(0)} - \allweights^* \|^2
   - \underbrace{\| \allweights^{(T+1)} - \allweights^* \|^2}_{\ge \, 0}
   + M \sum_{t=0}^T \alpha_t^2
\le \| \allweights^{(0)} - \allweights^* \|^2
   + M \sum_{t=0}^T \alpha_t^2.
\]
Taking $T \to \infty$, the right-hand side remains finite
by the step-size condition $\sum \alpha_t^2 < \infty$, so
\begin{equation}
\label{eq:weighted_gaps_finite}
\sum_{t=0}^\infty \alpha_t \Delta_t < \infty.
\end{equation}
Since $\Delta_t \ge 0$, $\sum \alpha_t = \infty$, and $\sum \alpha_t \Delta_t < \infty$,
we conclude $\liminf_{t \to \infty} \Delta_t = 0$:
if $\liminf \Delta_t > 0$, there would exist $\epsilon > 0$ and $T_0$ such that $\Delta_t \ge \epsilon$ for all $t \ge T_0$,
giving $\sum_{t \ge T_0} \alpha_t \Delta_t \ge \epsilon \sum_{t \ge T_0} \alpha_t = \infty$,
contradicting~\eqref{eq:weighted_gaps_finite}.
Hence $\liminf_{t \to \infty} \zweighted{\allweights^{(t)}} = \zCG$.
\end{proof}

\begin{corollary}[Extension to Inequality Constraints]
\label{cor:inequality_constraints}
Theorem~\ref{thm:quasigradient_convergence} extends to inequality constraints
$A^\partindex \vecx \le \vecb^\partindex$ (or $A^\partindex \vecx \ge \vecb^\partindex$)
by converting them to equalities with bounded slack variables.
For each group $\partindex \in [\numpartitions]$ with constraints $A^\partindex \vecx \le \vecb^\partindex$,
introduce $\vecs^\partindex \in \mathbb{R}^{\numconstraints{\partindex}}$ satisfying
\begin{equation}
\label{eq:slack_bounds}
\mathbf{0} \le \vecs^\partindex \le \vecb^\partindex - \min_{\vecx \in \X} A^\partindex \vecx
\end{equation}
and rewrite the constraints as $A^\partindex \vecx + \vecs^\partindex = \vecb^\partindex$.
Compactness of $\X$ ensures the upper bound in~\eqref{eq:slack_bounds} is finite,
so $\vecs^\partindex$ remains bounded.
The quasi-gradient becomes
$\qgradvector{\partindex}^{(t)} = \vecb^\partindex - A^\partindex \vecx^{(t)} - {\vecs^\partindex}^{(t)}$,
and the analysis carries through unchanged.
\end{corollary}

\section{Application to Two-Stage Stochastic Mixed-Integer Programs}
\label{sec:applications}

In this section, we review the application of Dantzig--Wolfe decomposition to 
two-stage SMIPs and adapt the aggregated constraint generation approach to this setting. 

Recall that the formulation \eqref{form:saa_nac} contains the consensus variable $\consensusvar$ and scenario variables
$(\vecx_\scenario,\vecy_\scenario)$ for $\scenario \in [\scenarios]$. The constraints preserved in
the RMP of the Dantzig--Wolfe relaxation include
the Cartesian product of the scenario feasible sets and the first-stage polyhedral constraints on the consensus
variables:
\[
\X = \mathcal{D}_1 \times \cdots \times \mathcal{D}_\scenarios \times \mathcal{Q}.
\]
The nonanticipativity constraints \eqref{eq:nac} correspond to the
linking constraints \eqref{eq:linking} in the general Dantzig--Wolfe formulation. 
To match the assumptions in Section \ref{sec:cg_background}, we assume
the sets $\mathcal{D}_k$ for $\scenario \in [\scenarios]$ and the set $\mathcal{Q}$  are all bounded, so that $\mathcal{X}$ is bounded.

In the SMIP setting, separability of the set $\X$ can be exploited to implement the convexification of $\X$ with the
convexification of each scenario set separately.
Let $\Pblock{\scenario} = \ext(\conv(\cgsubproblem))$ denote the extreme-point
set for scenario $\scenario$, and let $(\vecx_\vertex,\vecy_\vertex)$ denote the
first- and second-stage components of vertex $\vertex \in \Pblock{\scenario}$.
The resulting Dantzig--Wolfe relaxation is
\begin{mini!}
    {\consensusvar, \veclambda}
    {\sum_{\scenario=1}^\scenarios \scenprob \sum_{\vertex \in \Pblock{\scenario}}
        \Bigl(\vecc^\top \vecx_\vertex + \vecq_\scenario^\top \vecy_\vertex\Bigr) \lambda_{\scenario,\vertex}}
    {\label{form:cg_master}}{\zCG =}
    \addConstraint{\sum_{\vertex \in \Pblock{\scenario}} \vecx_\vertex\, \lambda_{\scenario,\vertex}}{= \consensusvar,
        \quad \forall\, \allscenarios \label{eq:cg_nac}}
    \addConstraint{\sum_{\vertex \in \Pblock{\scenario}} \lambda_{\scenario,\vertex}}{= 1,
        \quad \forall\, \allscenarios \label{eq:cg_convexity}}
	\addConstraint{\consensusvar \in \mathcal{Q}}
    \addConstraint{\lambda_{\scenario,\vertex}}{\ge 0, \quad \forall\, \allscenarios,\; \vertex \in \Pblock{\scenario}.}
\end{mini!}
The constraints $\consensusvar \in \mathcal{Q}$ are redundant in this formulation, but we include them because they are 
not necessarily  redundant in an aggregation-based relaxation.
For each scenario $\scenario \in [\scenarios]$, let
$\vecpi_\scenario \in \mathbb{R}^{\numcontvarsone}$ denote the
dual multipliers associated with the NACs \eqref{eq:cg_nac} for that scenario, and let
$\pi_{0,\scenario} \in \mathbb{R}$ denote the dual multiplier associated with
the corresponding convexity constraint \eqref{eq:cg_convexity}.
Because the RMP convexifies each $\cgsubproblem$ separately, pricing also
separates by scenario.
The pricing subproblem for scenario $\scenario$ is
\begin{equation}
\label{eq:cg_pricing_stochastic}
\min_{(\vecx,\vecy) \in \cgsubproblem}
\Bigl\{ \scenprob \bigl( \vecc^\top \vecx + \vecq_\scenario^\top \vecy \bigr) - \vecpi_\scenario^\top \vecx - \pi_{0,\scenario} \Bigr\}.
\end{equation}
The column-generation procedure for~\eqref{form:cg_master} follows the standard
scheme.
At each iteration, the RMP is solved over working column sets
$\Pblock{\scenario}' \subseteq \Pblock{\scenario}$ for all
$\scenario \in [\scenarios]$.
If a column with negative reduced cost is found as the solution to~\eqref{eq:cg_pricing_stochastic}, it is
added to the corresponding working set.
When no improving column exists for any scenario, the current RMP solution
attains $\zCG$.

A common assumption in two-stage SMIP is {\it relatively complete recourse}: for every solution $\vecx \in \mathcal{Q} \cap
(\firststagespace)$ which satisfies the
first-stage constraints, there is a second-stage solution $\vecy_k$ such that $(\vecx,\vecy_k) \in \mathcal{D}_k$ for every $k \in
[\scenarios]$. When this assumption is satisfied,
an initial set of columns that leads to a feasible RMP can be obtained by choosing some $\hat{\vecx} \in
\mathcal{Q}\cap(\firststagespace)$ (e.g., by solving the scenario subproblem
\eqref{eq:cg_pricing_stochastic} with $\vecpi_\scenario = 0$ for some $\scenario \in [\scenarios]$) and then solving the
second-stage problem (feasible by the relatively complete recourse assumption): 
\[ \min_\vecy \{ \vecq_\scenario^\top \vecy : (\hat{\vecx},\vecy) \in \cgsubproblem \} \]
to obtain $\hat{\vecy}_k$ for $k \in [\scenarios]$. Then setting $\mathcal{P}'_k = \{ (\hat{\vecx},\hat{\vecy}_k) \}$ for $k \in
[\scenarios]$ leads to a feasible RMP. 

We now specialize the aggregation framework of Section~\ref{sec:aggregate_constraints}
to two-stage SMIPs.
In this setting, the nonanticipativity constraints provide a natural aggregation structure:
grouping them by first-stage variable reduces the number of RMP linking constraints
from one per variable-scenario pair to one per first-stage variable.
We then develop partition refinement criteria and weight update rules that exploit this structure.

\subsection{Initial Group Structure of the Nonanticipativity Constraints}
\label{sec:nac_structure}

Recall from~\eqref{form:saa_nac} that the NACs require
$x_{\firststageindex,\scenario} = \cvarcomp_\firststageindex$
for every $\firststageindex \in [\totalfirststage]$
and scenario $\scenario \in [\scenarios]$.
Thus, the column-generation RMP contains
$m = (\totalfirststage) \cdot \scenarios$ linking constraints.
We partition these constraints by first-stage decision variable:
group $\constraintgroup{\firststageindex}$ contains the NACs that link the
$\firststageindex$th first-stage decision variables across all scenarios.
Applying the weighted aggregation~\eqref{eq:weighted_agg_general} to this group
replaces the $\scenarios$ NACs for each component with a single
surrogate constraint
\begin{equation}
\sum_{\scenario=1}^\scenarios w_{\firststageindex,\scenario} x_{\firststageindex,\scenario}
= \Bigl(\sum_{\scenario=1}^\scenarios w_{\firststageindex,\scenario}\Bigr) \cvarcomp_\firststageindex,
\qquad \forall \firststageindex \in [\totalfirststage].
\label{eq:weighted_nac}
\end{equation}
This construction yields $\numpartitions = \totalfirststage$ groups. 

Substituting the weighted NACs~\eqref{eq:weighted_nac} into the
Dantzig--Wolfe relaxation yields the following initial aggregated master problem:
\begin{mini!}
    {\consensusvar, \veclambda}
    {\sum_{\scenario=1}^\scenarios \scenprob \sum_{\vertex \in \Pblock{\scenario}}
        \Bigl(\vecc^\top \vecx_\vertex + \vecq_\scenario^\top \vecy_\vertex\Bigr) \lambda_{\scenario,\vertex}}
    {\label{form:agg_cg_master}}{}
    \addConstraint{\sum_{\scenario=1}^\scenarios w_{\firststageindex,\scenario}
        \sum_{\vertex \in \Pblock{\scenario}} x_{\vertex,\firststageindex}\, \lambda_{\scenario,\vertex}}
        {= \Bigl(\sum_{\scenario=1}^\scenarios w_{\firststageindex,\scenario}\Bigr) \cvarcomp_\firststageindex,
        \quad \forall\, \allfirststageindices \label{eq:agg_cg_nac}}
    \addConstraint{\sum_{\vertex \in \Pblock{\scenario}} \lambda_{\scenario,\vertex}}{= 1,
        \quad \forall\, \allscenarios \label{eq:agg_cg_convexity}}
		  \addConstraint{\consensusvar \in \mathcal{Q}}
    \addConstraint{\lambda_{\scenario,\vertex}}{\ge 0, \quad \forall\, \allscenarios,\; \vertex \in \Pblock{\scenario}.}
\end{mini!}

The corresponding pricing subproblem follows from the effective dual multipliers induced by the aggregated constraints.
Let $\mu_\firststageindex$ denote the dual multiplier associated with the aggregated constraint~\eqref{eq:agg_cg_nac}
for first-stage decision variable $\firststageindex \in [\totalfirststage]$.
Because $x_{\vertex,\firststageindex}$ from scenario $\scenario \in [\scenarios]$
enters~\eqref{eq:agg_cg_nac} with coefficient $w_{\firststageindex,\scenario}$,
the effective multiplier~\eqref{eq:effective_dual_general}
for the NAC of component $\firststageindex$ in scenario $\scenario$ is
$\pi_{\firststageindex,\scenario}^{\text{eff}} = \mu_\firststageindex\, w_{\firststageindex,\scenario}$.
Substituting this effective multiplier into the general pricing subproblem~\eqref{eq:weighted_pricing_general},
the pricing subproblem for scenario $\scenario \in [\scenarios]$ is
\begin{equation}
\label{eq:agg_pricing_stochastic}
\min_{(\vecx,\vecy) \in \cgsubproblem}
\Bigl\{ \scenprob \bigl( \vecc^\top \vecx + \vecq_\scenario^\top \vecy \bigr)
- \sum_{\firststageindex=1}^{\totalfirststage}
\mu_\firststageindex\, w_{\firststageindex,\scenario}\, x_\firststageindex
- \pi_{0,\scenario} \Bigr\},
\end{equation}
where $\pi_{0,\scenario}$ is the dual multiplier associated with the convexity
constraint~\eqref{eq:agg_cg_convexity}.

\subsection{Partition Refinement}
\label{sec:nac_variance_refinement}

Applying the partition refinement framework of Section~\ref{sec:refinement_general}
we begin with the initial partition used in \eqref{form:agg_cg_master} and use unit weights $\vecw_\scenario = \mathbf{1}$ for all
$\scenario \in [\scenarios]$.
At each iteration, the algorithm 
obtains a partition $\partition^{(t)}_\firststageindex$ of the scenario set
$[\scenarios]$
for each first-stage component $\firststageindex \in [\totalfirststage]$. 
This partition is initialized with the coarsest scenario partition
$\partition_\firststageindex^{(0)} = \{[\scenarios]\}$ for every
$\firststageindex \in [\totalfirststage]$.
A block $G \in \partition^{(t)}_\firststageindex$ represents the NACs $x_{\firststageindex,\scenario} = \cvarcomp_\firststageindex$ for $\scenario \in G$.
With unit weights, the resulting aggregated NACs have the structure 
\begin{equation}
\label{eq:partitioned_nac}
\sum_{\scenario \in G} x_{\firststageindex,\scenario}
=
|G|\,\cvarcomp_\firststageindex,
\qquad
\forall\, \firststageindex \in [\totalfirststage],\;
G \in \partition^{(t)}_\firststageindex ,
\end{equation}
which replaces \eqref{eq:agg_cg_nac} in formulation \eqref{form:agg_cg_master}.
At a general refinement iteration $t$, the number of linking constraints is
$\sum_{\firststageindex=1}^{\totalfirststage} |\partition_\firststageindex^{(t)}|$.

At iteration $t$ of the partition refinement algorithm, let  $(\zt,\veclambda^{(t)})$ be the solution
of the RMP (after column generation terminates) and let
\begin{equation}
\label{eq:itertxk}
 \vecx_{\scenario}^{(t)} = \sum_{\vertex \in \Pblock{\scenario}'} \vecx_\vertex
\lambda_{\scenario,\vertex}^{(t)}, \quad \scenario \in [\scenarios]. 
\end{equation}
We call a block
$G \in \partition_\firststageindex^{(t)}$ a candidate for refinement if at least
one scenario in the block violates its NAC, that is, if
$x_{\firststageindex,\scenario}^{(t)} \neq \cvarcomp^{(t)}_\firststageindex$ for some
$\scenario \in G$.
Among candidate blocks, we prioritize those with the largest within-block disagreement.
For a candidate block $G \in \partition_\firststageindex^{(t)}$, we measure this
disagreement by the mean squared NAC residual
\begin{equation}
\label{eq:variance_measure}
v_{\firststageindex,G}^{(t)}
=
\frac{1}{|G|}
\sum_{\scenario \in G}
\Bigl(
x^{(t)}_{\firststageindex,\scenario}
-
\cvarcomp^{(t)}_\firststageindex
\Bigr)^2 .
\end{equation}
At each refinement iteration, we rank candidate blocks over all $\firststageindex \in [\totalfirststage]$ by $v_{\firststageindex,G}^{(t)}$ and select a fixed fraction of the blocks with the largest disagreement for refinement.
Each selected block is split according to the sign of the individual NAC residuals.
For a selected block $G \in \partition_\firststageindex^{(t)}$, define
\begin{equation}
\label{eq:binary_partition}
G^- =
\Bigl\{
\scenario \in G :
x^{(t)}_{\firststageindex,\scenario}
<
\cvarcomp^{(t)}_\firststageindex
\Bigr\},
\qquad
G^+ =
\Bigl\{
\scenario \in G :
x^{(t)}_{\firststageindex,\scenario}
\ge
\cvarcomp^{(t)}_\firststageindex
\Bigr\}.
\end{equation}
In the next partition, each selected block $G$ is replaced by the 
sets $G^-$ and $G^+$.
This separates scenario copies below and above the value assigned to the consensus variable $\cvarcomp_\firststageindex$, preventing positive and negative residuals from canceling within the same aggregate in later iterations.

\subsection{Quasi-Gradient Weight Updates for Nonanticipativity Constraints}
\label{sec:nac_quasigradient}

We now specialize the quasi-gradient method of Section~\ref{sec:adaptive_weights}
to the NACs of two-stage SMIPs.
We first describe the extension using the NACs as implemented in~\eqref{eq:nac}. We then discuss an alternative
implementation of the NACs, which leads to a slightly different version of the quasi-gradient weight update method. 

\subsubsection{Consensus-Based Formulation}
\label{sec:nac_consensus}
For the consensus-variable NACs~\eqref{eq:nac},
let $(\zt,\veclambda^{(t)})$ be the solution
of the RMP (after column generation terminates) and let
$\vecx^{(t)}_{\scenario}, \scenario \in [\scenarios]$ be as defined in \eqref{eq:itertxk}.
The quasi-gradient~\eqref{eq:quasigradient_vector} then specializes to
\begin{equation}
\label{eq:quasigradient_nac}
g_{\firststageindex,\scenario}^{(t)} = \bar{x}^{(t)}_\firststageindex - x_{\firststageindex,\scenario}^{(t)},
\end{equation}
and the weight update~\eqref{eq:weight_update_vector} becomes
\begin{equation}
\label{eq:weight_update_nac}
w_{\firststageindex,\scenario}^{(t+1)} = w_{\firststageindex,\scenario}^{(t)} + \alpha_t
g_{\firststageindex,\scenario}^{(t)}, \quad \forall \firststageindex \in [n_1], \scenario \in [\scenarios]
\end{equation}
where $\alpha_t > 0$ satisfies the Robbins--Monro conditions~\eqref{eq:step_size_conditions}.

\subsubsection{Average-Based Formulation}
\label{sec:nac_average}
An alternative form of the NACs eliminates the consensus variables $\consensusvar$.
For each first-stage component, nonanticipativity is equivalent to requiring every scenario copy
to equal its average over all scenarios:
\begin{equation}
\label{eq:nac_average}
x_{\firststageindex,\scenario} = \sum_{\scenario'=1}^\scenarios p_{\scenario'} x_{\firststageindex,\scenario'},
\qquad \forall\, \firststageindex \in [\totalfirststage],\; \allscenarios.
\end{equation}
When using \eqref{eq:nac_average}, the consensus variable $\consensusvar$ is no longer needed and likewise the
constraints $\consensusvar \in \mathcal{Q}$ are excluded from the RMP.

For each first-stage component $\firststageindex \in [\totalfirststage]$,
aggregating the $\scenarios$ constraints in~\eqref{eq:nac_average}
with weights $\weightvector{\firststageindex}
= (w_{\firststageindex,1}, \ldots, w_{\firststageindex,\scenarios})^\top$ gives
\begin{equation}
\label{eq:weighted_nac_average}
\sum_{\scenario=1}^\scenarios w_{\firststageindex,\scenario}\, x_{\firststageindex,\scenario}
= \Bigl(\sum_{\scenario=1}^\scenarios w_{\firststageindex,\scenario} \Bigr)
  \sum_{\scenario'=1}^\scenarios p_{\scenario'} x_{\firststageindex,\scenario'},
\qquad \forall\, \firststageindex \in [\totalfirststage].
\end{equation}
This form is especially useful when the weights sum to zero.
If $\sum_{\allscenarios} w_{\firststageindex,\scenario} = 0$, then the right-hand side of~\eqref{eq:weighted_nac_average} vanishes, and the aggregate reduces to
\begin{equation}
\label{eq:weighted_nac_zerosum}
\sum_{\scenario=1}^\scenarios w_{\firststageindex,\scenario}\, x_{\firststageindex,\scenario} = 0,
\qquad \forall\, \firststageindex \in [\totalfirststage].
\end{equation}
The constraint in~\eqref{eq:weighted_nac_zerosum} is sparse whenever
the weight vector
$\weightvector{\firststageindex}
= (w_{\firststageindex,1}, \ldots, w_{\firststageindex,\scenarios})^\top$
is sparse, because only scenario copies with nonzero weights appear in the
left-hand side.
With $x^{(t)}_{\scenario}, \scenario \in [\scenarios]$  as defined in \eqref{eq:itertxk} after
solving the RMP at iteration $t$,
the quasi-gradient associated with~\eqref{eq:nac_average} is defined by the corresponding
residuals
\begin{equation}
\label{eq:quasigradient_zerosum}
g_{\firststageindex,\scenario}^{(t)}
= \sum_{\scenario'=1}^\scenarios p_{\scenario'} x^{(t)}_{\firststageindex,\scenario'}
  - x^{(t)}_{\firststageindex,\scenario} \quad \forall i \in [n_1], \scenario \in [\scenarios].
\end{equation}
Because these residuals sum to zero across scenarios,
\[
\sum_{\scenario=1}^{\scenarios} g_{\firststageindex,\scenario}^{(t)} = 0
\qquad \forall\, \firststageindex \in [\totalfirststage],
\]
the update~\eqref{eq:weight_update_nac} preserves the zero-sum property.
Specifically, if $\sum_{\allscenarios} w_{\firststageindex,\scenario}^{(0)} = 0$, then $\sum_{\allscenarios} w_{\firststageindex,\scenario}^{(t)} = 0$ for every iteration $t$.
Consequently, initializing Algorithm~\ref{alg:quasigradient_weighted} with zero-sum weights keeps the aggregated NAC in the simplified form~\eqref{eq:weighted_nac_zerosum} throughout the algorithm.


In our preliminary numerical study, we found that using the average-based NACs initialized with $w_{\scenario}^{(0)} = 0$
yielded better performance than using the consensus-based formulation in Algorithm~\ref{alg:quasigradient_weighted}.

\section{Computational Study}
\label{sec:computational_results}

In this section we report results of experiments evaluating the proposed ACG methods and benchmark algorithms.
Section~\ref{sec:test_problems} describes the test problems, Section~\ref{sec:exp_design} defines the experimental goals and performance
metrics, and
Section~\ref{sec:algorithm} gives the implementation details.
In Section \ref{sec:exact_bound_sensitivity} we investigate the impact of the proportion of groups to refine on the
partition refinement method. 
Sections~\ref{sec:exact_bound_experiments} and~\ref{sec:fast_bound_experiments}
then report two sets of experiments: exact recovery of the nonanticipativity
bound $\znac$ and fast lower-bound generation under limited work budgets.

\subsection{Test Problems}
\label{sec:test_problems}

The experiments use instances from two SMIP problems.
The first test problem is the Stochastic Generalized Assignment Problem (SGAP) which is based on an application
motivated by dynamic assignment of computing jobs to servers introduced in \cite{kothari2026}.
This problem class has pure binary decisions in both stages: first-stage variables
assign currently active jobs to machines, and second-stage variables assign jobs
under sampled future system trajectories.
The formulation
and description of instances are in Appendix~\ref{sec:sgap_formulation}. The test instances have between 45 and 215 first-stage variables, between 95 and 495
second-stage variables, and $K \in \{50,100\}$ scenarios.

The second test problem is the Stochastic Dynamic Capacity Acquisition and Assignment Problem (DCAP) 
drawn from SIPLIB \citep{siplib2015}, which models capacity acquisition and task assignment under uncertain demand.
Appendix~\ref{sec:dcap_formulation} gives the full formulation.
The first-stage decision variables are mixed-integer: variables to represent continuous resource capacity expansion
over multiple periods, and binary variables to record fixed charges when capacity is acquired.
The second stage has pure binary decisions for assigning tasks to the acquired capacity.
These instances have between 12 and 64 first-stage variables, between 27 and 240
second-stage variables, and between 50 and 500 scenarios.


\subsection{Methods and Performance Metrics}
\label{sec:exp_design}

We consider the following column generation based methods in our computational study.
\begin{description}[style=nextline]
\item[\variant{Standard}]
Column generation with the fully disaggregated NAC system, using simplex to solve
each RMP.

\item[\variant{Barrier}]
The same fully disaggregated formulation as \variant{Standard}, but with each
RMP solved by barrier with crossover disabled.
Interior-point RMP solves are a known form of dual stabilization in column
generation because they tend to return better-centered dual solutions
\citep{rousseau2007interior,gondzio2013new}.
This variant isolates that stabilization effect while keeping the RMP 
formulation unchanged.

\item[\variant{Partition}]
ACG with unit weights and the variance-based refinement rule of Section~\ref{sec:nac_variance_refinement}. At each outer
iteration, column generation is run to convergence on the current partition before violated blocks are refined.  

\item[\variant{Dyn-SAR}]
The dynamic separation of aggregated rows method \citep{costa2022stabilized}, adapted to aggregate NACs. Unlike \variant{Partition}, \variant{Dyn-SAR} retains previously generated aggregated rows while adding disaggregated rows for selected violations.
We solve the RMP with Gurobi's default LP algorithm to
leverage warm starts after rows are added.

\item[\variant{Quasi-gradient}]
Uses a fixed NAC partition and updates only the aggregation weights, as
described in Section~\ref{sec:nac_quasigradient}. 
\end{description}

For additional context, we also include the following non column-generation methods in our study:

\begin{description}
\item[\variant{FW--PH}]
The Frank--Wolfe progressive hedging method \cite{boland2018combining}; see Appendix~\ref{sec:ph_algorithm}. This method computes $\znac$ from the Lagrangian side: it maintains inner approximations of the scenario feasible-set convex hulls and updates scenario multipliers to enforce nonanticipativity.

\item[\variant{Subgradient}]
Applies subgradient optimization directly to the Lagrangian dual described in
Section~\ref{sec:general_cg_intro}.
\end{description}
\variant{FW--PH} has similarities to column generation because it also iteratively constructs an inner approximation of
the convex hull of each scenario set, but differs in that it does not solve a RMP to perform dual
updates.
We include \variant{Subgradient} as a comparison due to its similarity to \variant{Quasi-gradient} -- the weight updates
are similar except that \variant{Quasi-gradient} solves a RMP to inform the updates.

Table~\ref{tab:metrics} defines the performance metrics used in the computational study.
Unless otherwise stated, reported values are shifted geometric means with shift
parameter 1, computed over the relevant instance set.
The exceptions are the instance count metrics (\variant{Converged} and \variant{Reach}) and \variant{NAC\%}, for which we report the arithmetic mean. 
\begin{table}[ht]
  \centering
  \caption{Metrics reported in the computational results.}
  \label{tab:metrics}
  \begin{tabular}{lp{0.72\linewidth}}
    \toprule
    Metric & Description \\
    \midrule
    \variant{Converged}    & Number of instances on which the method satisfies its convergence criterion within the work-unit budget, reported as ``$c/N$'' \\
    \variant{Tot-W}         & Total Gurobi work units consumed \\
    \variant{RMP-W}       & Work units spent solving RMPs \\
    \variant{Sub-W}          & Work units spent solving scenario pricing subproblems \\
    \variant{Iters}        & Number of pricing iterations \\
    \variant{Refs}         & Number of partition refinement rounds \\
    \variant{NAC\%}        & Percentage of NACs represented explicitly in the RMP, relative to the full NAC system \\
    \variant{Columns}      & Total number of pricing columns added to the RMP\\
    \variant{Reach}        & Number of instances on which the method reaches a prescribed optimality-gap threshold within the work-unit budget \\
    \bottomrule
  \end{tabular}
\end{table}

We report computational effort in Gurobi work units rather than wall-clock or CPU time.
Work units provide a deterministic measure of solver effort for a fixed model and parameter setting, whereas wall-clock
and CPU time depend on the hardware and system load.
The experiments were run on the Center for High Throughput Computing (CHTC) at the
University of Wisconsin--Madison~\cite{chtc}.
Because this is a heterogeneous set of machines, work units provide a more consistent basis for comparing methods.
The Work columns in the tables should therefore be interpreted as solver effort, analogous to CPU time but less
hardware-dependent. Total work for a method (\variant{Tot-W}) is the sum of work units spent solving all problems in the method,
including the RMP (\variant{RMP-W}) and column generation subproblems (\variant{Sub-W}).
On an individual instance \variant{Tot-W} will be the sum of
\variant{RMP-W} and \variant{Sub-W}. However, due to the use of geometric means, this sum property is not preserved when
reporting aggregate results over sets of instances.

\subsection{Implementation Details}
\label{sec:algorithm}

All optimization problems are solved with Gurobi 11.0.
SGAP runs use four Gurobi threads. On the DCAP test instances, in preliminary experiments we found that the relationship between CPU time and Gurobi work units was less reliable when using four
Gurobi threads. Thus, we restrict Gurobi to one thread on the DCAP test instances.
Unless otherwise stated, RMPs are solved with the barrier
method (\texttt{Method=2}) and crossover disabled.
The exceptions are \variant{Standard}, which uses simplex, and
\variant{Dyn-SAR}, which uses Gurobi's default LP algorithm to leverage warm
starts after row additions.
We add a candidate column whenever its reduced cost is below $-\epsilon$, with
$\epsilon = 10^{-6}$.
Each scenario pricing problem is solved as a MIP with relative optimality gap
$0.1\%$.

All column-generation variants use the same initialization.
We first solve a single-scenario problem for an arbitrary scenario $\onescenario \in [\scenarios]$, obtaining a feasible first-stage solution $\startx$ and a corresponding recourse solution.
We then fix $\vecx = \startx$ and solve the recourse problem $\hat{\vecy}_{\scenario} = \argmin\{\vecq_{\scenario}^\top
\vecy_{\scenario} : (\vecx, \vecy_{\scenario}) \in \cgsubproblem,\; \vecx = \startx\}$ for each remaining scenario $\scenario \in [\scenarios] \setminus \{\onescenario\}$.
The resulting initial column for scenario $\scenario$ is $(\startx,\hat{\vecy}_{\scenario})$.
Because these columns share the same first-stage solution $\startx$, they satisfy the NACs and provide a feasible
initial RMP for both the disaggregated and aggregated formulations.

Stopping rules are method-specific.
For refinement-based ACG variants, a block is considered for disaggregation only if it contains a residual exceeding $\epsilon_{\mathrm{NAC}} = 10^{-3}$ in absolute value.
We use the same refinement rule for \variant{Partition} and \variant{Dyn-SAR}, as determined based on experiments in
reported in Section \ref{sec:exact_bound_sensitivity}.
For \variant{FW--PH}, we stop when the maximum first-stage disagreement across scenarios is at most $\epsilon_{\mathrm{FW\text{-}PH}} = 10^{-2}$.
Also in \variant{FW--PH}, we use parameter $\gamma = 1.0$ for SGAP and $\gamma = 20$ for DCAP (each of these chosen
after initial experiments).

Across all column-generation variants, we maintain a single persistent RMP model and modify it incrementally rather than rebuilding it between solves.
This saves model building time and, when solved with simplex, allows the solver to warm-start each re-optimization from the retained model state.
During partition refinement, we remove and re-add only the aggregated
constraints associated with refined groups; all other aggregated constraints,
columns, and convexity constraints remain unchanged.
For \variant{Quasi-gradient}, we update the coefficients of the aggregated NACs
in place after each iteration.
We use the average-based formulation of the NACs and initial weight vector $\vecw^{(0)} = 0$.

\subsection{Effect of the Refinement Fraction}
\label{sec:exact_bound_sensitivity}

In the partition refinement algorithm, we use a refinement fraction parameter $\rho$ the determines the fraction of
candidate blocks to disaggregate at
each outer iteration.
Smaller values of $\rho$ keep the RMP more compact, whereas larger
values refine more groups at each step and increase the number of linking
constraints more quickly.
We test $\rho \in \{0.1, 0.25, 0.5, 0.75, 1\}$.
The SGAP experiments use a work-unit budget of $7200$, and the DCAP experiments use a
work-unit budget of $3600$.
All other settings are held at the defaults in Section~\ref{sec:algorithm}.

\begin{table}[h]
\centering
\caption{Sensitivity to the refinement fraction $\rho$ on the SGAP subset.}
\label{tab:sgap_vp_sweep}
\begin{tabular}{lrrrrrrr}
\toprule
$\rho$ & \variant{Converged} & \variant{Tot-W} & \variant{RMP-W} & \variant{Sub-W} & \variant{Iters} & \variant{Refs} & \variant{NAC\%} \\
\midrule
0.1 & 9/12 &  146.4 & 3.5 & 145.4 & 62.8 & 25.9 &  \textbf{6.9} \\
0.25 & 9/12 &  131.0 & 3.8 & 129.7 & 55.3 & 13.7 &  7.8 \\
0.5 & 9/12 & \textbf{117.8} & 4.2 & \textbf{116.4} & \textbf{48.5} &  9.1 &  8.7 \\
0.75 & 9/12 &  130.2 & 5.1 & 128.4 & 52.6 &  7.0 &  9.2 \\
1.0 & 9/12 &  162.7 & 8.6 & 159.2 & 66.7 &  \textbf{6.5} &  9.8 \\
\bottomrule
\end{tabular}
\end{table}

Table~\ref{tab:sgap_vp_sweep} presents the results for the SGAP instances. 
The choice of $\rho$ affects the amount of work required to reach the bound. Smaller values
split fewer violated blocks at each refinement step, producing the most compact
final RMP but requiring more refinement rounds. Larger values refine more
aggressively, reducing the number of rounds but increasing the size
of the RMP more quickly. Splitting many groups at once adds many linking
constraints, so later column-generation phases can require more
pricing iterations. The best SGAP performance is obtained at the intermediate
value $\rho = 0.5$, which gives the lowest total work and the fewest pricing
iterations while still using only $8.7\%$ of the full NAC system. We therefore
use $\rho = 0.5$ as the default refinement fraction for SGAP in the remaining
experiments.

\begin{table}[h]
\centering
\caption{Sensitivity to the refinement fraction $\rho$ on the DCAP test set.}
\label{tab:dcap_vp_sweep}
\begin{tabular}{lrrrrrrr}
\toprule
$\rho$ & \variant{Converged} & \variant{Tot-W} & \variant{RMP-W} & \variant{Sub-W} & \variant{Iters} & \variant{Refs} & \variant{NAC\%} \\
\midrule
0.1 & 8/10 &  383.1 & 67.4 & 165.4 & 229.0 & 141.9 & \textbf{36.3} \\
0.25 & 8/10 &  242.3 & 48.9 & 111.9 & 153.3 &  56.7 & 38.4 \\
0.5 & 8/10 &  181.2 & 37.7 &  88.4 & 118.4 &  27.2 & 39.7 \\
0.75 & 8/10 & \textbf{160.8} & \textbf{34.1} & \textbf{82.2} & \textbf{108.6} &  17.1 & 39.7 \\
1.0 & 8/10 &  174.8 & 39.2 &  88.2 & 117.3 & \textbf{14.0} & 41.7 \\
\bottomrule
\end{tabular}
\end{table}

The DCAP results in Table~\ref{tab:dcap_vp_sweep} show a similar pattern.
The main difference is that DCAP benefits from more aggressive refinement.
Compared with SGAP, DCAP requires a larger fraction of the full NAC system to recover $\znac$.
A possible reason is that DCAP includes continuous first-stage variables, which allow scenario-specific first-stage solutions to differ by small amounts across many NACs.
Enforcing agreement may therefore require explicitly representing more NACs than in SGAP, where the first-stage assignment variables are all binary.
The best DCAP performance is obtained at $\rho = 0.75$, which gives the lowest
total work while keeping the master only modestly larger than with smaller
values of $\rho$.
We therefore use $\rho = 0.75$ as the default refinement fraction for DCAP in the remaining experiments.
Taken together, the experiments suggest that moderate refinement fractions
perform consistently well.
Values in the range $\rho \in [0.5,0.75]$ are either best or close to best on
both test problems.

\subsection{Exact-Bound Experiments}
\label{sec:exact_bound_experiments}

We next evaluate the performance of the methods in terms of the work required to compute the nonanticipativity bound $\znac$.
Each method is run to its own (near-optimal) convergence criterion.
We compare both the total work required to certify the bound and the cumulative
fraction of instances certified across the work-budget range. We exclude methods \variant{Quasi-gradient}
and \variant{Subgradient} in this study because they lack convergence checks that certify (near) attainment of the bound
$\znac$. 

Table~\ref{tab:convergence_exact}
reports convergence, work, and final-RMP statistics on the SGAP and DCAP test
sets.
Runs that do not reach $\znac$ within the work-unit budget are counted as not
converged and are charged the work-unit budget as their work in the mean calculations. For each test problem, the four
column generation based methods (\variant{Standard},\variant{Barrier},\variant{Dyn-SAR},\variant{Partition}) are grouped
together with \variant{FW--PH} reported last to emphasis that it is not a column generation method.

\begin{table}[h]
\centering
\caption{Convergence to the nonanticipativity bound $\znac$ and final RMP problem structure on the SGAP and DCAP test sets.}
\label{tab:convergence_exact}
\setlength{\tabcolsep}{5pt}
\begin{tabular}{lrrrrrrrr}
\toprule
Method & \variant{Converged} & \variant{Tot-W} & \variant{RMP-W} & \variant{Sub-W} & \variant{Iters} & \variant{Refs} & \variant{NAC\%} & \variant{Columns} \\
\midrule
\multicolumn{9}{l}{\emph{SGAP}} \\
\variant{Standard} &  9/24 & 407.9 & 141.0 & 184.5 &  84.9 &    --- & 100   & 4194.3 \\
\variant{Barrier}  & 19/24 & 265.1 &  45.8 & 227.0 & 158.4 &    --- & 100   & 7242.8 \\
\variant{Dyn-SAR}     & 21/24 &  62.7 &   8.1 &  57.2 &  40.3 &    \textbf{9.1} &  11.9  &  \textbf{1099.9} \\
\variant{Partition}   & 21/24 &  54.5 &   1.9 &  53.9 &  39.1 &    \textbf{9.1} &  \textbf{6.3}  & 1125.9 \\
\midrule
\variant{FW--PH}      & 21/24 &  \textbf{44.7} &   \textbf{1.3} &  \textbf{44.1} &  \textbf{24.4} &    --- & --- & --- \\
\midrule
\midrule
\multicolumn{9}{l}{\emph{DCAP}} \\
\variant{Standard} &  6/10 &  734.8 & 315.5 & 339.2 & 247.2 &    --- & 100   & 21196.4 \\
\variant{Barrier}  &  6/10 &  427.1 & 153.7 & 132.5 & 182.9 &    --- & 100   & 19045.9 \\
\variant{Dyn-SAR}     &  7/10 &  638.4 & 362.9 & 103.2 & 133.1 &   \textbf{17.4} & 81.3  & \textbf{12481.8} \\
\variant{Partition}   &  \textbf{8/10} &  \textbf{170.8} &  \textbf{40.0} &  \textbf{86.2} & \textbf{118.9} &   18.0 &  \textbf{40.7}  & 12580.3 \\
\midrule
\variant{FW--PH}      &  7/10 & 1691.2 & 315.7 & 407.2 & 677.2 &    --- & --- & --- \\
\bottomrule
\end{tabular}
\end{table}

The SGAP results show that \variant{FW--PH} has the lowest average total work,
but aggregation closes much of the gap between fully disaggregated column
generation and \variant{FW--PH}. \variant{Partition} certifies $\znac$ on $21$ of $24$
instances, matching the convergence count of \variant{FW--PH} and
\variant{Dyn-SAR}, and requires the least work among all column-generation
methods. Relative to \variant{Barrier}, which already benefits from
interior-point stabilization, \variant{Partition} reduces total work from
$265.1$ to $54.5$, RMP work from $45.8$ to $1.9$, and pricing iterations
from $158.4$ to $39.1$. Thus, aggregation improves performance in two ways: it
makes each RMP solve computationally less expensive, and it also reduces the
number of pricing iterations. The reduction in pricing iterations suggests that
aggregation also stabilizes the dual multipliers used by the pricing
subproblems.

The fully disaggregated column-generation baselines are limited by the size of
the NAC system. \variant{Standard} carries one linking constraint per NAC
throughout the algorithm, so the RMP is large from the first
iteration and the pricing subproblems are driven by a high-dimensional vector of
NAC multipliers.
\variant{Barrier} improves substantially over simplex-based column
generation because barrier without crossover returns better-centered dual
solutions \cite{rousseau2007interior,gondzio2013new}. However, it leaves the
NAC system fully disaggregated.
The lower pricing-iteration count of \variant{Standard} should not be
interpreted as faster convergence, because many simplex runs terminate early
after exhausting the work budget.

The comparison between \variant{Partition} and \variant{Dyn-SAR} shows that the
way aggregated rows are updated is important. Although both methods benefit from
aggregation, \variant{Partition} obtains a smaller final RMP: the number of NAC's added is only 
$6.3\%$ of the full NAC system, compared with $11.9\%$ for \variant{Dyn-SAR}.
This difference comes from replacing parent aggregates by their children rather
than retaining all previously generated aggregates. As a result,
\variant{Partition} reduces RMP effort from $8.1$ to $1.9$ and slightly
reduces total work, while generating essentially the same number of columns.

Overall, the SGAP results show that ACG closes most of the performance gap
between fully disaggregated column generation and \variant{FW--PH}. Although
\variant{FW--PH} has the lowest average total work on SGAP, \variant{Partition}
matches its convergence count and remains close in total work.

On the DCAP test set, no method converges on all ten instances, but
\variant{Partition} is the only method that certifies $\znac$ on $8$ of them.
It also has the lowest total work, RMP effort, subproblem effort, and
pricing-iteration count. 
The fully disaggregated baselines remain limited by the NAC system.
\variant{Barrier} reduces total work relative to \variant{Standard}, but the RMP work remains and the number of
iterations is larger than \variant{Partition}.

The comparison with \variant{Dyn-SAR} shows that aggregation is most
effective when the RMP remains compact. \variant{Dyn-SAR} improves the
convergence count relative to the fully disaggregated baselines, but it retains
previously generated aggregates while adding rows for selected violations. On
DCAP, this row-accumulation strategy produces a much denser final RMP:
the number of disaggregated constraints added by \variant{Dyn-SAR} is $81.3\%$ of the size of the full NAC system, compared with $40.7\%$ for \variant{Partition}. This larger RMP 
leads to substantially higher RMP effort despite a similar number of
generated columns.


\variant{FW--PH} is also less competitive on DCAP than on SGAP. It converges on
$7$ instances, but requires the largest total work and roughly five times as
many iterations as \variant{Partition}. One possible explanation is scaling:
DCAP has both continuous capacity-acquisition variables and binary fixed-charge
indicators in the first stage, so a single quadratic penalty parameter may not
balance disagreement across components as effectively as it does for the binary
assignment variables in SGAP. In contrast, \variant{Partition} does not require
a penalty parameter and performs more consistently on DCAP.

\subsection{Lower-Bound Experiments}
\label{sec:fast_bound_experiments}

We now study how quickly the methods produce high-quality lower bounds, rather than computational effort to (nearly)
solve the relaxation.
This setting is relevant when the 
dual bound is being used for bounding the optimality gap or within a branch-and-bound algorithm.

We compare the lower-bound performance of all methods. For each instance, we measure gaps
relative to the strongest lower bound obtained by any method within the work
budget (this strongest lower bound is equal to $\znac$ except for instances in which no method solved it within the work
limit). A method reaches threshold $\tau$ when its current lower bound is within
a relative factor of $\tau$ to that best bound.
If a method does not reach a threshold, it is charged the full work budget for
that instance.
Table~\ref{tab:lower_bound_quality}
reports the number of instances that reach each threshold and the shifted
geometric mean of the required work.

\begin{table}[h]
\centering
\caption{Lower-bound quality on the SGAP and DCAP test sets.}
\label{tab:lower_bound_quality}
\begin{tabular}{lrrrrrrrr}
\toprule
 & \multicolumn{2}{c}{$\tau = 10\%$} & \multicolumn{2}{c}{$\tau = 5\%$} & \multicolumn{2}{c}{$\tau = 1\%$} & \multicolumn{2}{c}{$\tau = 0.1\%$} \\
Method & \variant{Reach} & \variant{Tot-W} & \variant{Reach} & \variant{Tot-W} & \variant{Reach} & \variant{Tot-W} &
\variant{Reach} & \variant{Tot-W} \\
\midrule
\multicolumn{9}{l}{\emph{SGAP}} \\
\variant{Standard}        & 10 & 233.0 &  9 & 242.1 &  9 &  261.6 &  9 &  263.4 \\
\variant{Barrier}         & 17 & 253.9 & 17 & 265.4 & 17 &  284.7 & 17 &  314.6 \\
\variant{Dyn-SAR}            & \textbf{24} &   5.0 & \textbf{24} &  16.3 & 21 &   41.9 & 20 &   53.7 \\
\variant{Partition}          & \textbf{24} &   5.2 & \textbf{24} &  15.9 & 22 &   36.9 & 21 &   47.9 \\
\variant{Quasi-gradient}     & \textbf{24} &  \textbf{4.1} & \textbf{24} &  13.5 & 20 &   53.7 & 19 &   78.6 \\
\midrule
\variant{FW--PH}             & \textbf{24} &   4.5 & \textbf{24} &  \textbf{10.2} & \textbf{24} &  \textbf{17.9} & \textbf{23} &  \textbf{30.4} \\
\variant{Subgradient}        & 17 &  58.0 & 11 & 251.0 &  7 & 1316.4 &  6 & 2069.1 \\
\midrule
\midrule
\multicolumn{9}{l}{\emph{DCAP}} \\
\variant{Standard}        & \textbf{10} &   1.6 & \textbf{10} &   1.6 & 6 & 480.8 & 6 & 508.8 \\
\variant{Barrier}         & \textbf{10} &   1.6 & \textbf{10} &   1.6 & 6 & 220.0 & 6 & 245.2 \\
\variant{Dyn-SAR}            & \textbf{10} &   1.6 & \textbf{10} &   1.6 & \textbf{10} &  27.3 & \textbf{8} & 264.4 \\
\variant{Partition}          & \textbf{10} &   \textbf{1.6} & \textbf{10} &   \textbf{1.6} & \textbf{10} &  21.5 & \textbf{8} &  \textbf{68.7} \\
\variant{Quasi-gradient}     & \textbf{10} &   \textbf{1.6} & \textbf{10} &   \textbf{1.6} & \textbf{10} &  22.5 & \textbf{8} & 169.5 \\
\midrule
\variant{FW--PH}             & \textbf{10} &   \textbf{1.6} & \textbf{10} &   \textbf{1.6} & \textbf{10} &  \textbf{14.1} & \textbf{8} & 252.4 \\
\variant{Subgradient}        & 8 &   3.8 & 8 &   3.8 & 8 &  21.7 & 7 & 326.1 \\
\bottomrule
\end{tabular}
\end{table}

These results indicate that \variant{Quasi-gradient} is competitive with the other column generation methods at loose
thresholds, $\tau = 10\%$ and $\tau=5\%$. 
\variant{Quasi-gradient} requires slightly less work to reach these thresholds than the best of the other column
generation methods on the SGAP instances, and all column generation methods are equivalent at these thresholds on the
DCAP instances (this is because the initialization already yields a bound within $5\%$ on these instances).
However, the reach of \variant{Quasi-gradient} 
falls at tighter thresholds, indicating that weight updates alone do not
consistently recover the strongest bounds most quickly. The
refinement-based methods start slightly more slowly, but maintain better reach
at tight thresholds by enforcing nonanticipativity at a finer level. Between
the two refinement-based methods, \variant{Partition} outperforms
\variant{Dyn-SAR}, consistent with the exact-bound experiments and with its more
compact row-management strategy.

The comparison of \variant{Quasi-gradient} with \variant{Subgradient} highlights the benefit of
deriving the residuals used in the update steps from the aggregated RMP. Although both methods use
residual information to update weights/duals, \variant{Quasi-gradient} updates aggregation weights
using information obtained from solving a compact RMP via column generation, whereas \variant{Subgradient} updates Lagrangian
multipliers without maintaining a RMP. As a result,
\variant{Quasi-gradient} reaches the loose thresholds on all instances, while
\variant{Subgradient} reaches these thresholds on fewer instances and requires much more work as the
threshold tightens.

\variant{FW--PH} gives the strongest performance on the SGAP instances at tight thresholds,
reaching  $\tau = 0.1\%$ on nearly all instances and requiring the least
work at the two tightest thresholds.
For DCAP, at $\tau = 1\%$, \variant{FW--PH} requires the least work, showing that it can
produce strong lower bounds quickly even when certifying $\znac$ remains costly.
This suggests a long-tail convergence pattern of \variant{FW--PH} on DCAP: it reaches
moderately tight bounds quickly, but requires substantially more work to close
the remaining gap.


\section{Conclusion}
\label{sec:cg_conclusion}

This paper proposes an aggregated column generation framework for computing
dual bounds for two-stage stochastic mixed-integer programs with integer
recourse. The framework replaces groups of nonanticipativity constraints in the
Dantzig--Wolfe master with weighted surrogate constraints, reducing both the
cost of RMP solves and the dimension of the NAC multiplier space
used in pricing.

We studied two mechanisms for strengthening the aggregated relaxation.
Partition refinement starts from a coarse aggregation and iteratively splits
groups with violated NACs until the standard column-generation bound is
recovered. The weight update method instead keeps the partition fixed and uses
residual information to improve the bound without adding linking constraints.

The computational study indicates that the aggregation technique achieves benefits both in terms of reduced effort to
solve the RMP and to reduce the number of iterations required to obtain good dual bounds. Partition refinement appears
most effective at closing the gap to the full column generation bound, whereas the weight update method achieves bounds
within 5\% of the exact bound quickly but is slower to achieve full convergence. The results clearly indicate that the
work used to solve the restricted size RMP in the weight update method pays off in leading to faster convergence
compared to 
the subgradient method applied to solving the Lagrangian dual.

Several directions remain for future work. Partition refinement and
weight updates could be combined with each other, or with other techniques for enhancing column generation. 
It may be possible to improve the quasi-gradient method
with different step-size rules that accelerate 
 convergence. Finally, the
framework could be applied beyond stochastic programs to general decomposable
optimization problems with large linking systems, where constraint aggregation
may similarly reduce RMP size and stabilize pricing.

\begin{ackblock}
Luedtke acknowledges funding from the Air Force Office of Scientific
Research under award FA9550-25-1-0277 and from the Office of Naval Research under award N000142612067.
\end{ackblock}

\coiblock

\bibliographystyle{spmpsci}      
\bibliography{refs}

\appendix
\appendixheadings

\section{SGAP Formulation and Instance Details}
\label{sec:sgap_formulation}

Let $\jobst$ be a set of jobs to be assigned in the first stage.
Let $\pmachines$ be the set of physical machines, and let $\capacity_\machine$ denote the capacity of machine $\machine \in \pmachines$.
We also include an overflow machine $m_0$ with unlimited capacity to ensure feasibility, and define $\allmachines = \pmachines \cup \{m_0\}$.
Assigning job $\job \in \jobst$ to machine $\machine \in \allmachines$ incurs cost $\cost^\timestep_{\jobmachine}$ and, for $\machine \in \pmachines$, consumes $\resource^\timestep_{\jobmachine}$ units of capacity.

The second stage describes the system over the lookahead horizon $\lookaheadset := \{2,\ldots,\horizon\}$ under $\scenarios$ equally likely scenarios.
For scenario $\scenario \in \scenarioset$ and epoch $\epoch \in \lookaheadset$, let $J_\epoch^\scenario$ denote the jobs active at epoch $\epoch$, and let $\randomresource^{\epoch,\scenario}_{\jobmachine}$ denote the resource required by job $\job$ if assigned to machine $\machine$.
Define $\jobssecondk = \bigcup_{\epoch \in \lookaheadset} J_\epoch^\scenario$ as the scenario-dependent set of jobs that are active in at least one epoch of the lookahead horizon under scenario $\scenario$.
The first-stage variable $x^\timestep_{\jobmachine}$ is binary and indicates whether current job $\job \in \jobst$ is assigned to machine $\machine \in \allmachines$.
For each scenario $\scenario$, the second-stage variable $y^\scenario_{\jobmachine}$ is binary and indicates whether $\job \in \jobssecondk$ is assigned to machine $\machine \in \allmachines$.
For persisting jobs $\job \in \jobsintersectk$, the binary variable $z^\scenario_{\jobmachine}$ indicates whether assigning $\job$ to machine $\machine$ in scenario $\scenario$ requires a relocation relative to the first-stage assignment.
Relocations of persisting jobs incur cost $\penaltycost^\scenario_{\jobmachine}$, while assignments of new horizon arrivals $\job \in \jobssecondk \setminus \jobst$ incur cost $\randomcost^\scenario_{\jobmachine}$.
The extensive form minimizes the first-stage assignment cost plus the expected second-stage relocation and assignment costs:
\begin{mini!}|s|[constraintFormat]
    {}{\sum_{\job \in \jobst} \sum_{\machine \in \allmachines} \cost^{\timestep}_{\jobmachine}\, x^{\timestep}_{\jobmachine}
       + \frac{1}{\scenarios} \sum_{\scenario \in \scenarioset} \Bigl(
           \sum_{\job \in \jobsintersectk} \sum_{\machine \in \allmachines} \penaltycost^{\scenario}_{\jobmachine}\, z^{\scenario}_{\jobmachine}
         + \sum_{\job \in \jobssecondk \setminus \jobst} \sum_{\machine \in \allmachines} \randomcost^{\scenario}_{\jobmachine}\, y^{\scenario}_{\jobmachine}
       \Bigr)}{\label{eq:sgap_ef}}{}
    \addConstraint{\sum_{\machine \in \allmachines} x^{\timestep}_{\jobmachine}}{= 1}{\quad \forall\, \job \in \jobst \label{eq:sgap_assign_fs}}
    \addConstraint{\sum_{\job \in \jobst} \resource^{\timestep}_{\jobmachine}\, x^{\timestep}_{\jobmachine}}{\leq \capacity_{\machine}}{\quad \forall\, \machine \in \pmachines \label{eq:sgap_cap_fs}}
    \addConstraint{\sum_{\machine \in \allmachines} y^{\scenario}_{\jobmachine}}{= 1}{\quad \forall\, \job \in \jobssecondk,\; \scenario \in \scenarioset \label{eq:sgap_assign_ss}}
    \addConstraint{y^{\scenario}_{\jobmachine}}{\leq x^{\timestep}_{\jobmachine} + z^{\scenario}_{\jobmachine}}{\quad \forall\, \job \in \jobsintersectk,\; \machine \in \allmachines,\; \scenario \in \scenarioset \label{eq:sgap_nac_link}}
    \addConstraint{\sum_{\job \in J_\epoch^\scenario} \randomresource^{\epoch,\scenario}_{\jobmachine}\, y^{\scenario}_{\jobmachine}}{\leq \capacity_\machine}{\quad \forall\, \machine \in \pmachines,\; \epoch \in \lookaheadset,\; \scenario \in \scenarioset \label{eq:sgap_cap_ss}}
    \addConstraint{x^{\timestep}_{\jobmachine}}{\in \{0,1\}}{\quad \forall\, \job \in \jobst,\; \machine \in \allmachines \label{eq:sgap_binary_fs}}
    \addConstraint{y^{\scenario}_{\jobmachine}}{\in \{0,1\}}{\quad \forall\, \job \in \jobssecondk,\; \machine \in \allmachines,\; \scenario \in \scenarioset}
    \addConstraint{z^{\scenario}_{\jobmachine}}{\in \{0,1\}}{\quad \forall\, \job \in \jobsintersectk,\; \machine \in \allmachines,\; \scenario \in \scenarioset.}
\end{mini!}
Constraints~\eqref{eq:sgap_assign_fs} assign each current job to exactly one machine, and constraints~\eqref{eq:sgap_cap_fs} enforce first-stage capacities.
Constraints~\eqref{eq:sgap_assign_ss} assign each second-stage job to exactly one machine in each scenario, while constraints~\eqref{eq:sgap_cap_ss} enforce machine capacities at every epoch of the lookahead horizon.
Constraints~\eqref{eq:sgap_nac_link} activate $z^\scenario_{\jobmachine}$ whenever assigning a persisting job $\job$ to machine $\machine$ in scenario $\scenario$ requires a relocation relative to its first-stage assignment.

For the computational study, we generate SGAP instances using a discrete-event simulation with stochastic arrivals and
service durations described in detail in \cite{kothari2026}.
Arrivals follow a Poisson process calibrated to an average workload of $80\%$ of total physical capacity, and job durations have a mean of five epochs.
Resource requirements are sampled from a bimodal distribution with per-epoch multiplicative noise, and assignment costs are sampled uniformly from $[10,20]$.
The overflow penalty is set to ten times the mean assignment cost.
Each instance discards a 20-epoch warmup period before the first-stage decision is recorded.
These simulation parameters determine the first-stage job set and the scenario-dependent lookahead job sets used in the computational study.

The formulation includes $m$ physical machines and one overflow machine with infinite capacity.
The overflow machine ensures relatively complete recourse, although using it
incurs a large penalty cost.
At the decision epoch, if $|\jobst|$ jobs are active, the first stage contains
$|\jobst|(m+1)$ binary assignment variables.
Since NACs are imposed on these first-stage assignment variables across all $K$
scenarios, the full NAC system contains $|\jobst|(m+1)K$ constraints.
Because $|\jobst|$ is realized from the simulation, it varies across random
seeds.
Table~\ref{tab:sgap_instance_sizes} therefore reports average model sizes over
the seeds.
The first-stage counts are computed from the simulated instances before
rounding, so they need not equal the product of the rounded values of
$|\jobst|$ and $m+1$ shown in the table.

The SGAP test set varies the number of physical machines $m \in \{5,7,9\}$,
the lookahead horizon $\horizon \in \{1,3\}$, the scenario count $K \in \{50,100\}$,
and the random seed in $\{0,1\}$.
Table~\ref{tab:sgap_instance_sizes} reports representative first- and second-stage instance sizes.

\begin{table}[ht]
  \centering
  \caption{SGAP instance sizes, averaged over seeds.}
  \label{tab:sgap_instance_sizes}
  \begin{tabular}{ccrrr}
    \toprule
    $m$ & $\horizon$ & $|\jobst|$ & $\numcontvarsone$ & Avg.\ $\numcontvarstwo$ \\
    \midrule
    $5$ & $1$ & $8$  & $45$  & $95$  \\
    $5$ & $3$ & $8$  & $45$  & $117$ \\
    $7$ & $1$ & $14$ & $116$ & $228$ \\
    $7$ & $3$ & $14$ & $116$ & $270$ \\
    $9$ & $1$ & $22$ & $215$ & $428$ \\
    $9$ & $3$ & $22$ & $215$ & $495$ \\
    \bottomrule
  \end{tabular}

  \smallskip
  {\footnotesize The number of second-stage variables can differ across scenarios; the reported value is the average over scenarios.}
\end{table}

\section{DCAP Formulation and Instance Details}
\label{sec:dcap_formulation}

The DCAP \cite{ahmed2003scenario} involves $m$ resources, $n$ tasks, and $T$ time periods.
The first-stage decision is the capacity acquisition schedule $\{x_{it}\}$
for resources $i \in [m]$ and periods $t \in [T]$,
with fixed-charge acquisition costs:
\[
f_{it}(x_{it}) = \alpha_{it}\, x_{it} + \beta_{it}\, u_{it}, \qquad u_{it} = \begin{cases} 1 & \text{if } x_{it} > 0 \\ 0 & \text{otherwise,}\end{cases}
\]
where $\alpha_{it}, \beta_{it} \geq 0$ are variable and fixed cost coefficients
and $u_{it} \in \{0,1\}$ is the corresponding binary indicator.

For a given acquisition schedule $x$, scenario $\scenario \in [\scenarios]$, and period $t \in [T]$,
define the period-$t$ recourse function
\begin{mini}
{}{\sum_{i=0}^{m} \sum_{j=1}^{n} c_{ijt}^\scenario\, y_{ij}}{\label{eq:dcap_second_stage}}{Q_t^\scenario(x) = }
  \addConstraint{\sum_{j=1}^{n} d_{jt}^\scenario\, y_{ij}}{\le \sum_{\tau=1}^{t} x_{i\tau},}{\quad \forall\, i \in [m]}
  \addConstraint{\sum_{i=0}^{m} y_{ij}}{= 1,}{\quad \forall\, j \in [n]}
  \addConstraint{y_{ij}}{\in \{0,1\},}{\quad \forall\, i \in \{0,\ldots,m\},\, j \in [n].}
\end{mini}
Here $c_{ijt}^\scenario$ is the assignment cost in period $t$ under scenario~$\scenario$, and
$d_{jt}^\scenario$ is the corresponding processing requirement.
The cumulative capacity constraint ensures assignments in period $t$
do not exceed the installed capacity available through period $t$.
A dummy resource $i = 0$ with high cost ensures complete recourse.
The two-stage objective is:
\begin{equation}
\label{eq:dcap_obj}
\min_{(x,u) \in \X} \Bigl\{ \sum_{t=1}^T \sum_{i=1}^m (\alpha_{it}\, x_{it} + \beta_{it}\, u_{it}) + \sum_{\scenario=1}^{\scenarios} \scenprob
\sum_{t=1}^T Q_t^\scenario(x) \Bigr\}.
\end{equation}
The first-stage space is
$\mathbb{R}_+^{mT} \times \{0,1\}^{mT}$,
and the NACs enforce $x_{it,s} = \cvarcomp_{it}$ and $u_{it,s} = v_{it}$
for all $i, t, s$, yielding $2mT \cdot K$ constraints in total.

Each DCAP instance is specified by the number of resources $m$, the number of tasks $n$, and the number of time periods $T$.
The first stage contains $mT$ continuous acquisition variables and $mT$ binary fixed-charge indicators, for a total of $2mT$ first-stage variables.
NACs are imposed on both sets of first-stage variables, so the full NAC system contains $2mT K$ constraints.

The DCAP test set is drawn from the SIPLIB DCAP family and includes four base problems at multiple scenario counts, for a total of ten instances.
The instance name \texttt{dcap$mnT$} encodes the number of resources $m$, tasks $n$, and time periods $T$.
Table~\ref{tab:dcap_instances} lists the base problems, dimensions, and scenario counts used in the experiments.

\begin{table}[ht]
  \centering
  \caption{DCAP test instances.}
  \label{tab:dcap_instances}
  \begin{tabular}{lcccrrc}
    \toprule
    Base problem & $m$ & $n$ & $T$ & $\numcontvarsone$ & $\numcontvarstwo$ & Scenarios $K$ \\
    \midrule
    dcap233 & $2$ & $3$ & $3$ & $12$  & $27$  & $\{200, 300, 500\}$ \\
    dcap342 & $3$ & $4$ & $2$ & $12$  & $32$  & $\{200, 300, 500\}$ \\
    dcap468 & $4$ & $6$ & $8$ & $64$  & $240$ & $\{50, 200\}$ \\
    dcap575 & $5$ & $7$ & $5$ & $50$  & $210$ & $\{50, 200\}$ \\
    \bottomrule
  \end{tabular}
\end{table}

\section{Progressive Hedging for SMIPs}
\label{sec:ph_algorithm}

Progressive hedging \cite{rockafellar1991scenarios} decomposes a stochastic program by scenario and iteratively penalizes deviations from a common first-stage solution.
Its standard convergence guarantees require convexity, so we use the FW--PH variant \cite{boland2018combining}, which computes Lagrangian dual bounds for SMIPs by maintaining inner approximations of the convex hull of each scenario subproblem's feasible region.
Algorithm~\ref{alg:ph} summarizes the implementation, which is a specialization of the method presented in \cite{boland2018combining}.

For each scenario $\scenario \in [\scenarios]$, FW--PH maintains a finite set $V_\scenario \subseteq \cgsubproblem$ of generated scenario-feasible solutions.
The convex hull $\conv(V_\scenario)$ provides an inner approximation of $\conv(\cgsubproblem)$.
We initialize these sets using the same scenario solutions generated for the column-generation initialization described in Section~\ref{sec:algorithm}.
We initialize $\lagrmult_\scenario^{(1)} = 0$ for all $\scenario \in [\scenarios]$ and set $\bar{\vecx}^{(0)} = \startx$.

At iteration $t$, FW--PH first expands the inner approximation for each scenario by solving a linearized scenario subproblem over the full feasible set $\cgsubproblem$:
\begin{equation}
\label{eq:ph_milp}
(\hat{\vecx}_\scenario^{(t)}, \hat{\vecy}_\scenario^{(t)})
\in
\argmin_{(\vecx,\vecy) \in \cgsubproblem}
\Bigl\{
\bigl(\vecc + \lagrmult_\scenario^{(t)}\bigr)^\top \vecx
+ \vecq_\scenario^\top \vecy
\Bigr\},
\qquad \allscenarios.
\end{equation}
The generated solution is added to $V_\scenario$, thereby expanding the inner approximation $\conv(V_\scenario)$.
FW--PH then computes the scenario solution used in the consensus update by solving a penalized quadratic problem over the current inner approximation:
\begin{equation}
\label{eq:ph_qp}
(\vecx_\scenario^{(t)}, \vecy_\scenario^{(t)})
\in
\argmin_{(\vecx,\vecy) \in \conv(V_\scenario)}
\Bigl\{
\bigl(\vecc + \lagrmult_\scenario^{(t)}\bigr)^\top \vecx
+ \vecq_\scenario^\top \vecy
+ \frac{\gamma}{2}
\bigl\|\vecx - \bar{\vecx}^{(t-1)}\bigr\|_2^2
\Bigr\},
\qquad \allscenarios.
\end{equation}
The quadratic term penalizes deviation from the previous aggregate first-stage solution $\bar{\vecx}^{(t-1)}$.
The set $\conv(V_\scenario)$ is represented using convex combination coefficients $\vecmu \in
\mathbb{R}_+^{|V_\scenario|}$ satisfying $\mathbf{1}^\top \vecmu = 1$.
After solving the scenario subproblems, FW--PH updates the consensus point and the
multipliers:
\[
\bar{\vecx}^{(t)}
=
\sum_{\allscenarios} \scenprob\, \vecx_\scenario^{(t)},
\qquad
\lagrmult_\scenario^{(t+1)}
=
\lagrmult_\scenario^{(t)}
+
\gamma\Bigl(\vecx_\scenario^{(t)} - \bar{\vecx}^{(t)}\Bigr),
\quad \allscenarios.
\]
The algorithm stops when the maximum infinity-norm disagreement between any scenario first-stage solution
$\vecx_\scenario^{(t)}$ and the probability-weighted average $\bar{\vecx}^{(t)}$ is at most the tolerance $\epsilon_{\mathrm{FW\text{-}PH}}$, or when the work limit is reached.
Under standard assumptions, FW--PH converges to the optimal Lagrangian dual bound $\zLD = \zCG$.

\begin{algorithm}[t]
\caption{Frank--Wolfe Progressive Hedging for SMIPs \cite{boland2018combining}}
\label{alg:ph}
\begin{algorithmic}[1]
\State \textbf{Input:} Penalty parameter $\gamma > 0$, tolerance
       $\epsilon_{\mathrm{FW\text{-}PH}} > 0$, and initial feasible solution
       $\startx$.
\State Initialize $V_\scenario \gets \{(\startx,\hat{\vecy}_\scenario)\}$,
       $\lagrmult_\scenario^{(1)} \gets 0$ for all $\allscenarios$, and
       $\bar{\vecx}^{(0)} \gets \startx$.
\For{$t = 1, 2, \ldots$}
    \For{each $\scenario \in [\scenarios]$}
        \State Solve MILP~\eqref{eq:ph_milp} to obtain
               $(\hat{\vecx}_\scenario^{(t)},\hat{\vecy}_\scenario^{(t)})$.
        \State $V_\scenario \gets
               V_\scenario \cup
               \{(\hat{\vecx}_\scenario^{(t)},\hat{\vecy}_\scenario^{(t)})\}$.
        \State Solve QP~\eqref{eq:ph_qp} over $\conv(V_\scenario)$ to obtain
               $(\vecx_\scenario^{(t)},\vecy_\scenario^{(t)})$.
    \EndFor
    \State $\bar{\vecx}^{(t)} \gets
           \sum_{\allscenarios} \scenprob\, \vecx_\scenario^{(t)}$.
    \If{$\max_{\allscenarios}
         \|\vecx_\scenario^{(t)}-\bar{\vecx}^{(t)}\|_\infty
         \leq \epsilon_{\mathrm{FW\text{-}PH}}$}
        \State \textbf{break}
    \EndIf
    \State $\lagrmult_\scenario^{(t+1)} \gets
           \lagrmult_\scenario^{(t)}
           + \gamma(\vecx_\scenario^{(t)}-\bar{\vecx}^{(t)})$
           for all $\allscenarios$.
\EndFor
\end{algorithmic}
\end{algorithm}

\end{document}